\documentclass[11 pt,reqno]{amsart}
 
\usepackage{amsmath,amsthm,amssymb,amscd,graphicx}
\usepackage[utf8]{inputenc}
\usepackage[T1]{fontenc}
\usepackage{bbm, dsfont}
\usepackage{color}
\usepackage{hyperref}
\usepackage{enumerate}
\usepackage{mathtools}
\usepackage{mathrsfs}
\usepackage{multirow}

\newtheorem{theorem}{Theorem}[section]

\newtheorem{lemme}[theorem]{Lemma}

\newtheorem{prop}[theorem]{Proposition}

\theoremstyle{definition}
\newtheorem{df}[theorem]{Definition}
\theoremstyle{remark}
\newtheorem{remark}[theorem]{Remark}

\numberwithin{equation}{section}

\renewcommand{\phi}{\varphi}

\def\C{\mathbb C}
\def\R{\mathbb R}
\def\T{\mathbb T}

\def\N{\mathbb N}

\def\E{\mathcal E}

\def\M{\mathcal M}

\def\eps{\varepsilon}
\def\dis{\displaystyle}    

\newcommand{\pa}[1]{\left(#1\right)}
\newcommand{\pac}[1]{\left[#1\right]}

\newcommand{\bra}[1]{\left\langle #1 \right\rangle}
\newcommand{\abs}[1]{\left|#1\right|}

\definecolor{gr}{rgb}   {0.,   0.69,   0.23 }
\definecolor{bl}{rgb}   {0.,   0.5,   1. }
\definecolor{mg}{rgb}   {0.85,  0.,    0.85}
\definecolor{yl}{rgb}   {0.8,  0.7,   0.}
\definecolor{or}{rgb}  {0.7,0.2,0.2}

\newcommand{\dd}{\mathrm{d}}

\renewcommand{\Re}{  {\mathfrak{Re}}  }

\newcommand{\ov}{  \overline  }

\DeclareRobustCommand{\rchi}{{\mathpalette\irchi\relax}}
\newcommand{\irchi}[2]{\raisebox{\depth}{$#1\chi$}}

\usepackage{graphicx} 
\usepackage{wrapfig}  
\usepackage{tikz}     
\usepackage{pgfplots}
\usepgfplotslibrary{fillbetween}
\usetikzlibrary{patterns}
\pgfplotsset{compat=1.16}

\begin{document}
\author{Valentin Schwinte}

\title[An Optimal Minimization Problem in the LLL]{An Optimal Minimization Problem in the Lowest Landau Level and Related Questions}

\thanks{Adress: CNRS, IECL, Université de Lorraine, F-54000 Nancy, France. Email: valentin.schwinte@univ-lorraine.fr}

\begin{abstract}
We solve a minimization problem related to the cubic Lowest Landau level equation, which is used in the study of Bose-Einstein condensation. We provide an optimal condition for the Gaussian to be the unique global minimizer. This extends previous results from P.~G\'erard, P.~Germain and L.~Thomann. We then provide another condition so that the second special Hermite function is a global minimizer.
\end{abstract}

\maketitle

\section*{Acknowledgement}
The author warmly thanks Laurent Thomann, Pierre Germain and Nicolas Rougerie for the numerous insightful discussions leading to this work, and for all the remarks and improvements they provided.

\tableofcontents

\renewcommand{\labelenumi}{(\roman{enumi}).}

\section{Introduction and Main Results.}
\subsection{Physical modelling and motivations}
A \emph{Bose-Einstein condensate} is a state of matter one can obtain by cooling a gas made of non relativistic and identical bosons without spin nor potential energy below a critical temperature, called Bose-Temperature. More precisely, below this temperature, the majority of bosons are in the ground state, and we observe a new state of matter. We call this phenomenon \emph{Bose-Einstein condensation}, and it has applications in superfluidity, superconductivity, and black holes modelling~\cite{ABD,Rou1}.

Many properties of Bose-Einstein condensates are studied in quantum physics~\cite{AB,ABD,BBCE1,CRY,FB,Ho1,Rou1,RSY2}. Notably, a rising question of the last decades is the formation of vortices in such a condensate. A vortex is, in dimension~$2$, a point where the density vanishes, that is to say a zero of the density of probability~$|u|^2$ of the solution~$u$. In the presence of rotation, for an isotropic harmonic trap in the limit of weak interactions between particles, two possible regimes can appear for the Bose-Einstein condensate. The first one, when the rotation frequency is much smaller than the trapping frequency, the condensate has at most one vortex, but when the rotation frequency nearly compensates the trapping frequency, the number of vortices grows to infinity, and the radius of the system diverges. It seems that, in the latter regime, the vortices organise in a lattice. Some of the previous properties can be extended to a non harmonic confinement, for instance a quadratic plus quartic potential~\cite{ABD,BR,Rou2,RSY2}.

Let us consider such a condensate, confine it with a harmonic field, and put it in rotation at a high velocity. Then its dynamics is described by the Lowest Landau Level~\eqref{lll} equation:
\begin{equation*}\label{lll}\tag{LLL}
    \left\{
        \begin{array}{ll}
        i\partial_tu = \Pi(|u|^2u), &(t,z) \in \R \times \C, \\
        u(0,\cdot) = u_0 \in \mathcal{E},&
        \end{array}
    \right.
\end{equation*}
where $$ \mathcal{E} = \{ u(z)=e^{-\frac{|z|^2}{2}}f(z) , \;f \;\text{holomorphic}\} \cap L^2(\C),$$
is the Bargmann-Fock space, and $\Pi$ is the orthogonal projector onto $\mathcal{E}$. This equation had been studied in~\cite{GGT} and its dynamical properties have first been studied from a mathematical point of view in~\cite{Nier1}. The existence of progressive waves and of unbounded trajectories for coupled Bose-Einstein condensate has been studied in~\cite{SchTho} and~\cite{Tho1}. Note that the Eq.~\eqref{lll} and its dynamics are included in the studies~\cite{FGH,GHT1} of the cubic resonant~(CR) equation.

The Eq.~\eqref{lll} is Hamiltonian, with the structure  
$$ \Dot{u} = -i \frac{\delta \mathcal{H}}{\delta \ov{u}},\qquad \Dot{\ov{u}} = i \frac{\delta \mathcal{H}}{\delta {u}}$$
and the Hamiltonian functional
$$ \mathcal{H}(u) = \frac14\int_{\C}|u|^4 \dd L,$$
where $L$ stands for the Lebesgue measure on $\C$.

The following symmetries, called respectively phase rotations, space rotations, and magnetic translations leave invariant the Hamiltonian~$\mathcal{H}$:
$$T_\gamma : u(z) \longmapsto e^{i\gamma}u(z) \qquad \gamma \in \mathbb{R},$$
$$L_\theta : u(z) \longmapsto u(e^{i\theta}z) \qquad \theta \in \mathbb{R},$$
$$R_\alpha : u(z) \longmapsto u(z+\alpha)e^{\frac12(\overline{z}\alpha-z\overline{\alpha})} \qquad \alpha\in\C.$$
Then by Noether theorem, they are related to invariant quantities by the flow of~\eqref{lll}, respectively to the mass~$M$, the angular momentum~$P$ and the magnetic momentum~$Q$:
$$M(u) = \int_\C |u(z)|^2\dd L(z), \qquad P(u) = \int_\C (|z|^2-1)|u(z)|^2\dd L(z), $$
$$Q(u) = \int_\C z|u(z)|^2\dd L(z).$$

The energy functional of a Bose-Einstein condensate in a harmonic trap at high speed rotation can be written as 
$$ \mathcal{G}_\mu(u) := 8 \pi \mathcal{H}(u) + \mu P(u) = 2\pi \int_\C|u(z)|^4 \dd L(z)  + \mu \int_\C(|z|^2-1)|u(z)|^2\dd L(z),$$
where $\mu >0$ is a parameter. We consider in this article the minimization problem,  for $\mu > 0$,
\begin{equation}\label{min_pb}\tag{$\star$}
	\min_{\substack{u \in \mathcal{E} \\ M(u) =1}} \mathcal{G}_\mu(u),
\end{equation}
the constraint $M(u)=1$ being the classical normalisation of quantum mechanics. The global minimizers of~\eqref{min_pb} (which exist, see Theorem~\ref{thm1} below) are $L^2(\C)$ stationary waves for the~\eqref{lll} equation, i.e. some particular solution. See Definition~\ref{def_stat} for a rigorous definition.


\subsection{Mathematical motivation}
First, we give the expression of the Gaussian:
$$ \phi_0(z) := \frac{1}{\sqrt{\pi }}e^{-\frac{|z|^2}{2}}.$$
In~\cite{GGT}, the authors study the minimization problem above, arising in particular in the analysis of stationary waves of the Eq.~\eqref{lll}. We recall their main results in Sect.~\ref{stats_waves}. More precisely, they consider for $\mu > 0$ the minimization problem
$$
\min_{\substack{u \in \mathcal{E} \\ M(u) =1}} \mathcal{G}_\mu(u) \qquad \mbox{with} \qquad \mathcal{G}_\mu(u) =  8 \pi \mathcal{H}(u) + \mu P(u),
$$
for which they prove optimal condition in the parameter~$\mu$ for some local minimizers, and study the global minimization problem: they prove (see~\cite[Proposition~7.4, Proposition~7.5 and Remark~7.6]{GGT}) that the Gaussian~$\phi_0$ is a strict \emph{local} minimizer (up to symmetries) \emph{if and only if} $\mu> \frac12$, and the \emph{unique} (up to symmetries) global minimizer of $\mathcal{G}_\mu$ for any $\mu \geq \sqrt{3}-1$.

One can ask the following question: 
\begin{center}
Is~$\phi_0$ the unique (up to symmetries) global minimizer for any $\mu > \frac12$ ? 
\end{center}
In the sequel, we shall answer this question, and shall be able to give a necessary and sufficient condition on~$\mu$ for the Gaussian~$\phi_0$ to be the unique global minimizer of~$\mathcal{G}_\mu$. Thanks to~\cite{Clerck-Evnin}, we use a suitable change of variables leading to the study of an infinite quadratic form. We prove the latter is positive using linear algebra tools. We will also address the cases $\mu = \frac{1}{2}$ and $ \mu < \frac{1}{2}.$


\subsection{Main results.}
In this subsection, we will describe the main results of the paper, and give some insights on the proofs.

\subsubsection{Optimal global minimizers}
Our first result concerns directly the minimization problem, and answers the question raised in the previous paragraph. We improve the bound $\mu \geq \sqrt{3}-1$ for~$\phi_0$ to be the unique (up to phase rotations) global minimizer of~\eqref{min_pb}, to the optimal condition $\mu > \frac{1}{2}$:
\begin{theorem}\label{th1.1}
	The Gaussian~$\phi_0$ is the~\emph{unique} global minimizer of~\eqref{min_pb} for any $\mu > \frac12$, up to phase rotations. This condition is optimal.
\end{theorem}
\begin{remark}
    The optimality is to understand in the following sense: $\phi_0$ is a strict local minimizer of~$\mathcal{G}_\mu$ \emph{if and only if} $\mu > \frac12$, so it cannot be global minimizer for any $\mu < \frac12$. The uniqueness is up to phase and space rotations, both giving the same family of functions. Furthermore,~$\phi_0$ is \emph{a} global minimizer for $\mu = \frac12$, but it is not unique: $\mathcal{G}_{\frac12}(\phi_0)=\mathcal{G}_{\frac12}(\phi_1)=\mathcal{G}_{\frac12}(\psi_b).$ See Theorem~\ref{case_mu_12} for the definitions of the functions~$\phi_1$ and~$\psi_b$.
\end{remark}
We give here the outline of the proof. First, we show that Theorem~\ref{th1.1} is equivalent to
\begin{theorem}\label{th1.2}
The functional
$$
B(u)=4\pi \mathcal H(u)+\frac{1}{4}   \pa{ M(u)P(u)-|Q(u)|^2     }-\frac12 M^2(u)
$$
is non-negative, where we recall the expressions of the quantities $\mathcal{H}, M, P$ and $Q$:
$$\mathcal{H}(u) = \frac14 \int_\C |u(z)|^4 \dd L(z), \qquad M(u) = \int_\C |u(z)|^2\dd L(z),$$
$$P(u) = \int_\C (|z|^2-1)|u(z)|^2\dd L(z), \qquad Q(u) = \int_\C z|u(z)|^2\dd L(z),$$
where~$L$ stands for the Lebesgue measure on~$\C$.
\end{theorem}
The main idea is then to study~$B$, which is a homogeneous functional, instead of directly studying the problem~\eqref{min_pb}. This idea comes from~\cite{Clerck-Evnin}, where the authors write~$B$ as an infinite quadratic form, which we need to prove that it is non-negative. The result is announced without analytical proof, but the authors have verified it numerically for the first terms. The difficulty lies in the degeneration of the functional~$B$, limiting the number of methods that can be used, and in particular, excluding perturbation methods.

\begin{remark}
	After writing this article, the author learned that this result was already known in a slightly different form \cite[equation (14)]{LewSei}. In addition, the equality between the energy in the space of the lowest Landau level and the energy of the gaussian $\phi_0$ (see below) is discussed in \cite{LSY}. The method used here to show the result is different and instructive. In particular, we will show that all $B^{(j)}$ terms are non-negative and not just their sum, as is the case in \cite{LewSei}. Moreover, the question of the optimality of the condition $\mu \geq \frac12$ and the transition to a vortex for $\mu = \frac12$ have no equivalent to the knowledge of the author.
\end{remark}

In~\cite{GGT}, the authors prove that~$\phi_0$ is the unique global minimizer of~$\mathcal{G}_\mu$ for any $\mu \geq \sqrt{3}-1$, by studying a different functional, which is 
$$ F_\mu(u)=8\pi \mathcal H(u)+M(u)(\mu P(u)-M(u)).$$
They write $\dis u = \sum_{n=0}^{+\infty} a_n\phi_n,$ with $(\phi_n)$ a Hilbertian basis of $\mathcal{E}$, they discard the majority of the terms, and study the resulting quadratic form. This method can be improved by discarding less terms, nevertheless, it should fail to give the optimal bound on~$\mu$ (which is~$\frac12$), due to real obstructions, such as the degeneration of~$B$. In fact, we will show in the sequel that~$F_{\frac12}$ is non-negative if and only if~$B$ is non-negative.

We also address the case $\mu = \frac12$:
\begin{theorem}\label{case_mu_12}
	Suppose $\mu = \frac12$. Then, the global minimizers of \eqref{min_pb} are, up to phase and space rotations, $ \phi_0, \phi_1, \psi_b$, where
        $$\phi_0(z) = \frac{1}{\sqrt{\pi }}e^{-\frac{|z|^2}{2}},$$
	$$\phi_1(z) := \frac{1}{\sqrt{\pi }}ze^{-\frac{|z|^2}{2}},$$
	and 
	$$ \psi_b(z) := \frac{e^{-\frac12\pa{\frac{b}{1+b^2}}^2}}{\sqrt{\pi(1+b^2)}}\pa{z-\frac{b(2+b^2)}{1+b^2}}e^{\frac{b}{1+b^2}z-\frac{|z|^2}{2}}, \quad 0 \leq b < +\infty.$$
\end{theorem}
We stress that the family~$(\psi_b)$ links up~$\phi_0$ and~$\phi_1$: we have, for all $z\in\C,$ $ \psi_b(z) \xrightarrow[b \to 0]{ }\phi_1(z) $ and $ \psi_b(z) \xrightarrow[b \to +\infty]{ }-\phi_0(z). $

Then the following question arises: what happens in the case $ \mu < \frac12?$ We have, for any $\mu < \frac12$, 
\begin{align*}
    \mathcal{G}_\mu(\phi_1) &= \frac12 + \mu, \\
    \mathcal{G}_\mu(\psi_b) &= 1 + \pa{\mu - \frac{1}{2}} \frac{1}{(1+b^2)^2},\\
    \mathcal{G}_\mu(\phi_0) &= 1,
\end{align*}
so that, for any $\mu < \frac12$, 
$$ \mathcal{G}_\mu(\phi_1) < \mathcal{G}_\mu(\psi_b) < \mathcal{G}_\mu(\phi_0).$$
Moreover, by~\cite{GGT}, the function~$\phi_1$ is a strict local minimizer for any $\mu \in (5/32,1/2)$, so that it is a good candidate for being a global minimizer of~\eqref{min_pb} for some $\mu < \frac12$. We will prove in Sect.~\ref{section_cas_phi1} the following result:
\begin{theorem}\label{min_phi1}
	There exists $\mu_0 \in [5/32,1/2)$ such that, for all $\mu \in (\mu_0,1/2), \, \phi_1$ is the unique global minimizer of~\eqref{min_pb} up to space and phase rotations.
\end{theorem}
The fact that such a~$\mu_0$ satisfies $\mu_0 > \frac{5}{32}$ comes from \cite[Proposition~7.4,~$(iii)$]{GGT}: $\phi_1$ is a strict local minimizer if and only if $\frac{5}{32} < \mu < \frac12.$ Moreover, this implies that there exists a~$\mu_1$ satisfying: for all $\mu \in (\mu_1,1/2), \, \phi_1$ is the unique (up to phase and space rotations) global minimizer, and~$\phi_1$ is not a global minimizer for any $\mu < \mu_1$. Then, by the classification result~\cite[Theorem~6.1]{GGT}, for any $\mu < \mu_1$, any global minimizer has an infinite number of zeros. To summarise, we state:
\begin{theorem}\label{resume_th}
	Consider for $\mu>0$ the minimization problem
        \begin{equation*}
            \min_{\substack{u\in\E \\ M(u)=1}} \mathcal{G}_\mu(u) \qquad \text{with} \qquad \mathcal{G}_\mu =  8\pi \mathcal{H} + \mu P.
        \end{equation*}
    \begin{enumerate}
        \item  For any $ \mu>0$, there \emph{exists} a global minimizer.
        \item  For $\mu>\frac 12$, the Gaussian $\phi_0$ is the \emph{unique global} minimizer of $\mathcal{G}_\mu$ over $\{ u \in \mathcal{E}, M(u) = 1 \}$, up to phase and space rotations.
        \item For $\mu = \frac12$, the global minimizers are, up to phase and space rotations, $\phi_0, \phi_1$ and the family $\{\psi_b\}_{b\geq 0}$.
        \item There exists $\mu_0 \in [\frac{5}{32}, \frac12)$ such that:
        \begin{enumerate}
            \item for all $\mu \in (\mu_0,\frac12),$ $\phi_1$ is the unique global minimizer, up to phase and space rotations,
            \item for all $\mu < \mu_0$, \emph{any global} minimizer has an infinity of zeros.
        \end{enumerate}
        \item  For $\mu \in (0,\frac{5}{32})$, \emph{any local or global} minimizer of $\mathcal{G}_\mu$ has an infinity of zeros.
        \end{enumerate}
\end{theorem}
\begin{remark}
	The case $\mu = \mu_0$ is still an open problem. It remains unclear what to expect, but deserves further investigation that might lead to further work.
\end{remark}

\subsubsection{Relation to the semi-classical approach}
We also study the number of zeros of minimizers for the~\eqref{min_pb} problem. In particular, it is physically relevant to investigate whether or not a solution has an finite or infinite number of zeros, since a zero represents the position of a vortex for the condensate. We now turn into semi-classical coordinates, to be able to compare our results to those of mathematical physics literature, in particular~\cite{ABN}.
So, we fix $h \in (0,1)$ a small parameter. We work on the space 
$$ \mathcal{E}_h = \{ u(z)=e^{-\frac{|z|^2}{2h}}f(z) , \;f \;\text{holomorphic}\} \cap L^2(\C),$$
and define the energy 
$$ E_h(u) = \int_{\C} \pa{|z|^2|u(z)|^2+\frac{Na\Omega_h^2}{2}|u(z)|^4}\dd L(z), \qquad N,a>0,\; \Omega_h=\sqrt{1-h^2}.$$
We consider the mass $\dis M(u) = \int_\C|u(z)|^2 \dd L(z),$ and the minimization problem
\begin{equation}\label{min_pb_h}
\min_{\substack{u \in \mathcal{E}_h \\M(u) =1}} E_h(u). 
\end{equation}
This problem is first addressed in \cite[Theorem 1.2]{ABN}. In this article, the authors find conditions on the parameter $h$ for which the solution of \eqref{min_pb_h} has infinitely many zeros. In \cite{GGT}, the authors improve this last result, by giving an better range of the parameter $h$, and without any condition on the Lagrange multiplier associated to the problem.
We are here able again to weaken their conditions. Our result states as follow:
\begin{theorem}\label{th_semi_classical}
	We denote by $\phi_{0,h}$ and $\phi_{1,h}$ the first two Hermite functions in semi-classical coordinates, that is to say
    \begin{align*}
        \phi_{0,h}(z) &:= \frac{1}{\sqrt{\pi h}}e^{-\frac{|z|^2}{2h}},\\
        \phi_{1,h}(z) &:= \frac{1}{\sqrt{\pi h}}\frac{z}{\sqrt{h}}e^{-\frac{|z|^2}{2h}}.
    \end{align*}
    We set $\kappa_\infty = \frac{5}{32}$ and $\kappa_0 = \frac{1}{2}.$
    \begin{enumerate}
        \item Assume that 
        \begin{equation*}
            h > \sqrt{\kappa_0 \frac{Na\Omega_h^2}{4\pi}}.
        \end{equation*}
        Then the Gaussian $\phi_{0,h}$ is the unique global minimizer of \eqref{min_pb_h}, up to symmetries. Furthermore, this bound is optimal, and we can compute
        \begin{equation*}
            E_h(\phi_{0,h}) = \frac{Na\Omega_h^2}{4\pi h} + h.
        \end{equation*}
        \item Assume that 
        \begin{equation*}
            h < \sqrt{\kappa_\infty \frac{Na\Omega_h^2}{4\pi}}.
        \end{equation*}
        Then any minimizer (local or global) of the problem \eqref{min_pb_h} has an infinite number of zeros.
        \item There exists $\kappa_1 \in (\kappa_\infty,\kappa_0)$ such that, for any $h$ satisfying
        \begin{equation*}
            \sqrt{\kappa_1 \frac{Na\Omega_h^2}{4\pi}} < h < \sqrt{\kappa_0 \frac{Na\Omega_h^2}{4\pi}},
        \end{equation*}
        the function $\phi_{1,h}$ is the unique global minimizer of \eqref{min_pb_h}, up to symmetries. We compute :
        \begin{equation*}
            E_h(\phi_{1,h}) = \frac{Na\Omega_h^2}{8\pi h} + 2h.
        \end{equation*}
    \end{enumerate}
\end{theorem}
This Theorem is the analogous of Theorem~\ref{resume_th} above, with the rescaling $u(z) = \sqrt{h} v(\sqrt{h}z)$. Item~$(ii)$ is already obtained in~\cite{GGT}, but here we weaken the condition of~$(i)$, which was $\kappa_0 = \sqrt{3}-1 > \frac12$. The new condition on $\kappa_0$ is optimal, since $\phi_{0,h}$ is not a local minimum for $h < \sqrt{\kappa_0 \frac{Na\Omega_h^2}{4\pi}}.$ Item $(iii)$ is a novelty.

\subsection{Plan of the paper}

The paper is organised as follows. In Sect.~\ref{section_general} we recall some general background about the Bargmann-Fock space and Eq.~\eqref{lll}. We also give a brief state of art on stationary waves for \eqref{lll}, and recall some notations. Section \ref{section_proof} is devoted to the proof of Theorem~\ref{th1.1}. In Sects.~\ref{section_cas_1_2} and~\ref{section_cas_phi1}, we aim at giving more insights about the minimization problem in general, and prove both Theorems \ref{case_mu_12} and \ref{min_phi1}.


\section{General Results on  \texorpdfstring{\eqref{lll}}{} and Notations}\label{section_general}
In this section, we recall general results and notations for the \eqref{lll} equation. We mainly refer to \cite{GGT} for proofs.

\subsection{Structure of \texorpdfstring{\eqref{lll}}{}}
First recall the cubic \eqref{lll} equation:
\begin{equation*}\tag{LLL}
    \left\{
        \begin{array}{ll}
        i\partial_tu = \Pi(|u|^2u), &(t,z) \in \R \times \C, \\
        u(0,\cdot) = u_0 \in \mathcal{E},&
        \end{array}
    \right.
\end{equation*}
for which
$$ \mathcal{E} = \{ u(z)=e^{-\frac{|z|^2}{2}}f(z) , \;f \;\text{holomorphic}\} \cap L^2(\C).$$
This space is called the Bargmann-Fock space, and $\Pi$ is its orthogonal projector. It is a Banach space for the usual $L^2(\C)$ scalar product, for which we give the Hilbertian basis $(\phi_n)_{n\in\N}$ (see~\cite[Proposition 2.1]{Zhu}) defined by
$$ \phi_n(z) = \frac{1}{\sqrt{\pi n!}}z^ne^{-\frac{|z|^2}{2}}.$$
These functions are called the special Hermite functions.
We denote $L^{2,1}(\C)$ the weighted $L^2(\C)$ space given by the norm 
$$\|f\|_{L^{2,1}} = \|\bra{z}f(z)\|_{L^2}, \qquad \bra{z} = \sqrt{1+|z|^2},$$
and then define 
$$ L_{\mathcal{E}}^{2,1} = \{ u(z)=e^{-\frac{|z|^2}{2}}f(z) , \;f \;\text{holomorphic}\} \cap L^{2,1}(\C).$$

We also define, for any $p \geq 1$ the space
$$ L_{\mathcal{E}}^{p} = \{ u(z)=e^{-\frac{|z|^2}{2}}f(z) , \;f \;\text{holomorphic}\} \cap L^{p}(\C),$$
which is a Banach space. For any real numbers $p,q$ such that $1 \leq p \leq q \leq +\infty$, the space $L_{\mathcal{E}}^{p}$ is embedded in $L_{\mathcal{E}}^{q}$, with the optimal constant (see~\cite{Carlen}) being, for all $u \in \mathcal{E},$
\begin{equation}\label{eq_carlen}
\pa{\frac{q}{2\pi}}^{1/q} \|u\|_{L^q(\C)} \leq \pa{\frac{p}{2\pi}}^{1/p} \|u\|_{L^p(\C)}.
\end{equation}
We call the latter inequality Carlen's estimates.

We recall that the operator $\Pi$ is the orthogonal projector onto $\E$, we compute its Kernel $K$:
$$ \forall (z,\xi)\in\C\times \C, \qquad  K(z,\xi) = \sum_{n=0}^{+\infty} \phi_n(z) \overline{\phi_n}(\xi) = \frac{1}{\pi}e^{\overline{\xi}z}e^{-|\xi|^2/2}e^{-|z|^2/2},$$
which gives the explicit formula for the projector $\Pi$
\begin{equation*}
    [\Pi u](z)=\frac{1}{\pi}e^{-\frac{|z|^2}{2}} \int_{\C} e^{\overline{\xi}z-\frac{|\xi|^2}{2}}u(\xi)\dd L(\xi),
\end{equation*}
where $L$ represents the Lebesgue measure on $\C$.

Throughout the paper, we use the notation $z=x+iy$ for complex variables and use the classical notations $\partial_z$ and $\partial_{\ov{z}}$ of complex analysis:
$$\partial_z=\frac12(\partial_x-i\partial_y), \qquad \partial_{\ov{z}}=\frac12(\partial_x+i\partial_y).$$ 

Note that the \eqref{lll} equation is induced by the Hamiltonian 
\begin{equation*}
\mathcal{H}(u) := \frac14\int_\C|u|^4\dd L, \quad u\in\E.
\end{equation*}
We define the following symmetries, called respectively phase rotations, space rotations, and magnetic translations: 
$$T_\gamma : u(z) \longmapsto e^{i\gamma}u(z) \qquad \gamma \in \mathbb{T},$$
$$L_\theta : u(z) \longmapsto u(e^{i\theta}z) \qquad \theta \in \mathbb{T},$$
$$R_\alpha : u(z) \longmapsto u(z+\alpha)e^{\frac12(\overline{z}\alpha-z\overline{\alpha})} \qquad \alpha\in\C.$$
These three symmetries leave invariant the Hamiltonian $\mathcal{H}$, then by Noether theorem, they are related to invariant quantities by the flow of \eqref{lll}, respectively to the mass $M$, the angular momentum $P$ and the magnetic momentum $Q$:
$$M(u) = \int_\C |u(z)|^2\dd L(z), \qquad P(u) = \int_\C (|z|^2-1)|u(z)|^2\dd L(z), $$
$$Q(u) = \int_\C z|u(z)|^2\dd L(z), $$
for some $u\in L_{\mathcal{E}}^{2,1}.$ Observe that the mass $M$ is invariant by the three symmetries, whereas the angular momentum $P$ is invariant by phase and space rotations, but not magnetic translations:

\begin{equation}\label{action_trans_P}
	P(R_\alpha u)=P(u) - 2\Re(\overline{\alpha}Q(u))+|\alpha|^2M(u).
\end{equation}
Finally, the magnetic momentum is only invariant by phase rotations. Space rotations and magnetic translations act on $Q$ as follows,
\begin{equation}\label{action_rot_Q}
	Q(L_\phi u)=e^{-i\phi}Q(u),
\end{equation}
and
\begin{equation}\label{action_trans_Q}
	Q(R_\alpha u) =Q(u) - \alpha M(u).
\end{equation}

If we decompose $u\in\mathcal{E}$ in the basis $(\phi_n)_{n\in\N}$, that is to say
\begin{equation*}
u = \sum_{n=0}^{+\infty} a_n\varphi_n,
\end{equation*}
we can give another representation of the conserved quantities:
\begin{equation}\label{MP_sum}
M(u) = \sum_{n=0}^{+\infty} |a_n|^2, \qquad P(u) = \sum_{n=0}^{+\infty} n|a_n|^2,
\end{equation}
\begin{equation*}
Q(u) = \sum_{n=0}^{+\infty} \sqrt{n+1}a_n\ov{a_{n+1}}.
\end{equation*}
Moreover, the Hamiltonian can be written as 
\begin{equation*}
    \mathcal{H}(u) = \frac{1}{8\pi}\sum_{\substack{ k,\ell,m,n \geq 0 \\ k + \ell = m + n }} \frac{(k + \ell)!}{2^{k+\ell} \sqrt{k! \ell ! m ! n!}}\ov{a_k} \ov{a_\ell}a_ma_n =   \frac{1}{8\pi}\sum_{j =0}^{+\infty} \frac{j!}{2^{j }}\abs{\sum _{n+p=j} \frac{a_na_p}{\sqrt{n!p!}}}^2.
\end{equation*}

We also define the function $\phi_n^\alpha$ by
\begin{equation*}
\phi^\alpha_n(z) = R_{-\ov{\alpha}}(\phi_n)(z) = \frac{1}{\sqrt{\pi n!}}(z-\ov{\alpha})^ne^{-\frac{|z|^2}{2}-\frac{|\alpha|^2}{2}+\alpha z}.
\end{equation*}

Finally, the well-posedness of~\eqref{lll} is first studied by F. Nier in~\cite{Nier1} and then by P. Gérard, P. Germain and L. Thomann in~\cite{GGT}. The latter authors proved global well-posedness in $\E$ and in $L_{\mathcal{E}}^{2,1}$, with conservation of the quantities $M,P,Q$. They also studied local and global well-posedness for more general spaces. We refer to~\cite{GGT} for details.

\subsection{A brief state of art on stationary waves}\label{stats_waves}
\begin{df}\label{def_stat}
	We call $M$-stationary wave in $\E$ a solution of the Eq.~\eqref{lll} taking the form 
	$$u(t) = e^{-i\lambda t}u_0, \quad \text{with} \quad  u_0\in \E,\lambda\in\R.$$
	We call $MP$-stationary wave in $\E$ a solution of the Eq.~\eqref{lll} taking the form 
	$$u(t) = e^{-i\lambda t}u_0(e^{-i\mu t}\cdot), \quad \text{with} \quad u_0\in \E,\lambda,\mu\in\R.$$
\end{df}
The study of such stationary waves is relevant for physical applications. In particular, we are interested in the number of zeros of such solutions. We have the following classification of stationary waves with finite number of zeros (see \cite[Theorem 6.1]{GGT}):
\begin{enumerate}
	\item The $M$-stationary waves in $\E$ with a finite number of zeros and unit mass are given, modulo space and phase rotation, by $\phi_n^\alpha(z)e^{-i\lambda t}$, with $n\in\N, \alpha\in \C$, and where $ \lambda = \dfrac{(2n)!}{\pi(n!)^22^{2n+1}}$.
	\item Besides the $\phi_n^\alpha$, the $MP$-stationary waves in $\E$ with a finite number of zeros and unit mass are given, modulo space and phase rotation, by $\psi_b(e^{-i\mu t}z)e^{-i\lambda t}$, where
$$\psi_b(z)=\frac{e^{-\frac12\pa{\frac{b}{1+b^2}}^2}}{\sqrt{\pi(1+b^2)}}\pa{z-\frac{b(2+b^2)}{1+b^2}}e^{\frac{b}{1+b^2}z-\frac{|z|^2}{2}},$$
with $b>0$, $\mu = -\dfrac{1}{8\pi}$, and  $\dis \lambda =\frac{1}{8\pi(1+b^2)}\pa{2b^2+1+\frac{b^2}{1+b^2}}.$
\end{enumerate}

Recall, for $\mu >0$, the minimization problem
$$
\min_{\substack{u \in \mathcal{E} \\ M(u) =1}} \mathcal{G}_\mu(u) \qquad \mbox{with} \qquad \mathcal{G}_\mu(u) =  8\pi \mathcal{H}(u)  + \mu P(u).
$$
Then we have the following results, concerning both local and global minimizers:
\begin{enumerate}
	\item \cite[Proposition 7.5 (i)]{GGT} For any $\mu > 0$, there exists a global minimizer.
	\item Any local or global minimizer of $\mathcal{G}_\mu$ is a stationary wave.
	\item \cite[Proposition 7.4 (i)]{GGT} The function $\phi_0$ is \emph{a strict local} minimizer if and only if $\mu > \frac12$.
	\item \cite[Proposition 7.5 (ii)]{GGT} The function $\phi_0$ is \emph{the unique global} minimizer at least for $\mu \geq \sqrt{3}-1$.
	\item \cite[Proposition 7.4 (ii)]{GGT} The function $\phi_1$ is \emph{a strict local} minimizer if and only if $\frac{5}{32} < \mu < \frac{1}{2}$.
	\item \cite[Proposition 7.4 (iii), Proposition 7.5 (iii)]{GGT} If  $0 < \mu < \frac{5}{32}$, then any local or global minimizer has an infinite number of zeros.
	\item \cite[Proposition 7.4 (iv)]{GGT} The function $\phi_k$, with $k\geq2$, is not a local minimizer for any value of $\mu > 0$.
	\item \cite[Proposition 7.4 (v)]{GGT} The function $\psi_b$, with $b>0$, is not a local minimizer for any value of $\mu \neq \frac12$.
\end{enumerate}

To summarise, we draw the following chart:
\begin{center}
\begin{tabular}{|| c | c | c | c | c ||} 
    \hline
     & Case $0 < \mu < \frac{5}{32}$ & Case $ \frac{5}{32} < \mu < \frac{1}{2}$ & Case $ \frac{1}{2} < \mu < \sqrt{3}-1 $ & Case $ \mu \geq \sqrt{3}-1$ \\
    \hline
	Local & Any local minimizer & $\phi_1$ is a strict & $\phi_0$ is a strict  & $\phi_0$ is a strict\\
	minimization & has an infinite & local minimizer &  local minimizer &  local minimizer\\
	problem & number of zeros & & & \\
    \hline
	Global & Any global minimizer & Partial answer: & $\phi_0$ is a strict & $\phi_0$ is a strict\\
	minimization & has an infinite & Theorem \ref{min_phi1} & global minimizer &  global minimizer\\
	problem & number of zeros \cite{GGT} & & Theorem \ref{th1.1} & \cite{GGT} \\
    \hline
\end{tabular}
\end{center}

In the sequel, we will complete the case  $ \frac{1}{2} < \mu < \sqrt{3}-1 $ for the global minimization. This is the result of Theorem \ref{th1.1}. We also address partially the case $\frac{5}{32} < \mu < \frac{1}{2}$, which is the result of Theorem \ref{min_phi1}. Furthermore, we treat the case $\mu = \frac12$ for the global minimization problem.


\section{Proof of Theorem \ref{th1.1}}\label{section_proof}
In this section, we will mainly prove the Theorem \ref{th1.1}, that we state here as follow:
\begin{theorem}  \label{thm1}
    Consider for $\mu>0$ the minimization problem
    \begin{equation*}
        \min_{\substack{u\in\E \\ M(u)=1}} \mathcal{G}_\mu(u) \qquad \text{with} \qquad \mathcal{G}_\mu =  8\pi \mathcal{H} + \mu P.
    \end{equation*}
\begin{enumerate}
    \item  For any $ \mu>0$, there exists a global minimizer of $\mathcal{G}_\mu$ over $\{ u \in \mathcal{E}, M(u) = 1 \}$.
    \item  For $\mu>\frac 12$, the Gaussian $\phi_0$ is the unique (up to symmetries) global minimizer of $\mathcal{G}_\mu$ over $\{ u \in \mathcal{E}, M(u) = 1 \}$.
    \item  For $\mu \in (0,\frac{5}{32})$, the global minimizer of $\mathcal{G}_\mu$ has an infinity of zeros.
    \end{enumerate}
\end{theorem}

\begin{remark}
Items $(i)$ and $(iii)$ have already been proved in \cite{GGT}. The novelty is item $(ii)$ with the optimal condition $\mu >1/2$ (in \cite{GGT} the condition was $\mu \geq \sqrt{3}-1$). 
\end{remark}

\begin{remark}\label{case_mu_0}
	With regard to item $(i)$, one can ask about the case $\mu = 0$, for which there is no global minimizer: suppose there exists a global minimizer $u$ for $\mu = 0$. Then, $u$ satisfies
$$ \forall n \in \N, \qquad 8\pi \mathcal{H}(u) = \mathcal{G}_0(u) \leq \mathcal{G}_0(\phi_n) = \frac{(2n)!}{2^{2n}n!^2} \xrightarrow[n \to +\infty]{ } 0,$$
due to Stirling estimate, so that $\mathcal{H}(u) = \frac14\|u\|_{L^4}^4 = 0$. Then $u=0$, which is not consistent with the constraint~$M(u) = 1$.
\end{remark}

\begin{remark}
    In \cite{GGT}, it is shown that~$\phi_0$ is the unique global minimizer of $\mathcal{G}_\mu$ among even functions in~$\E$ if $\mu > \frac12$. But the case of even functions is much simpler because it cancels most of the quadratic form's correlations between the coefficient in the basis $(\phi_n)$.
\end{remark}

We do not know entirely what happens in the case $5/32< \mu< 1/2$. The function~$\phi_1$ is a local minimizer, and a global minimizer for~$\mu$ close enough to~$\frac12$ (see Theorem~\ref{resume_th}), but we cannot obtain any explicit bound with our method.

\begin{remark}
Observe that
\begin{align*}
  \mathcal{G}_\mu(\phi_0) &= 1 \\
  \mathcal{G}_\mu(\phi_1) &= \frac{1}{2} + \mu \\
  \mathcal{G}_\mu(\psi_b) &= 1 + \pa{\mu - \frac{1}{2}} \frac{1}{(1+b^2)^2}.
\end{align*}
This implies in particular that $\mathcal{G}_\mu(\phi_0) =  \mathcal{G}_\mu(\phi_1) = \mathcal{G}_\mu(\psi_b) = 1$ if $\mu = \frac{1}{2}$, hence the minimizer is not unique when $\mu=\frac12$. We plot the values of $\mathcal{G}_\mu$ for values of $\mu$ for the functions $\phi_0,\phi_1$ and $\psi_b$ in Fig.~\ref{fig:graph1}. We also emphasise here that
\begin{equation*}
    \forall z \in \C, \qquad \psi_b(z) \xrightarrow[b \to 0]{ }\phi_1(z)
\end{equation*}
and 
\begin{equation*}
    \forall z \in \C, \qquad \psi_b(z) \xrightarrow[b \to +\infty]{ }-\phi_0(z). 
\end{equation*}
\end{remark}

\hspace{2cm}
\begin{center}
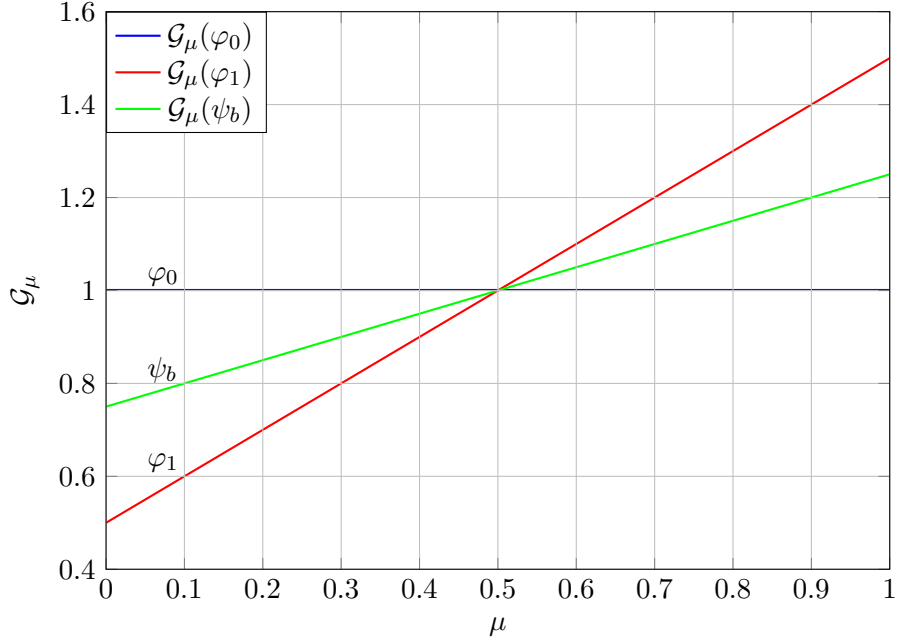
\begin{figure}[ht]
\begin{tikzpicture}
\begin{axis}[
	  xlabel={$\mu$},
        ylabel={$\mathcal{G}_\mu$},
        axis on top=true,
        xmin=0,
        xmax=1,
        ymin=0.4,
        ymax=1.6,
        height=9.0cm,
        width=12.0cm,
        grid, 
	 legend entries={},
	 legend style={at={(0,1)},anchor=north west}]
    ]
	\addplot [domain=0:1, samples=80, mark=none, thick, blue, name path=G0] {1.001+0*x};
		\addlegendentry{$\mathcal{G}_\mu(\phi_0)$}
	\addplot [domain=0:1, samples=80, mark=none, thick, red, name path=G1] {1/2+x};
		\addlegendentry{$\mathcal{G}_\mu(\phi_1)$}
	\addplot [domain=0:1, samples=80, mark=none, thick, green, name path=Gb] {1+(x-1/2)/(1+1^2)};
		\addlegendentry{$\mathcal{G}_\mu(\psi_b)$}
    \node [right, black] at (axis cs: 0.04,1.03) {$\phi_0$};
    \node [right, black] at (axis cs: 0.04,0.83) {$\psi_b$};
    \node [right, black] at (axis cs: 0.04,0.63) {$\phi_1$};
\end{axis}
\end{tikzpicture}
\caption{Graphs of $\mathcal{G}_\mu(u)$ for different stationary waves $u$. Here, we used $b=1$.} \label{fig:graph1}
\end{figure}
\end{center}

\subsection{Preliminaries: definition of some functionals and relation to the problem}

As in \cite[equation (23)]{Clerck-Evnin} we define the functional
  \begin{eqnarray*}
    B(u)&=&4\pi \mathcal{H}(u)+\frac{1}{4}\pa{M(u)P(u)-|Q(u)|^2     }-\frac12 M^2(u)\\
    &=&\frac12 \sum_{j=0}^{+\infty} \frac{j!}{2^j} \abs{\sum_{k=0}^{j}\frac{a_k a_{j-k}}{\sqrt{k ! (j-k)!}}}^2 +\frac{1}4  \sum^{+\infty}_{n,k=0}(n-2) \ov{a_n}\ov{a_k}{a}_n{a}_k \\
    & & \qquad \qquad \qquad \qquad \qquad \qquad \qquad \qquad \qquad -\frac14  \sum^{+\infty}_{n,k=0}\sqrt{n+1}\sqrt{k+1}\ov{a_{n+1}} \ov{a_k}{a}_n{a}_{k+1}.
\end{eqnarray*}

Notice that this functional is homogeneous. Furthermore, we have the following lemma:
\begin{lemme}\label{lemme_invariance_B}
	The quantity $B$ is invariant by the symmetries of the equation:
	$$ \forall \phi \in \R, \forall u \in L^{2,1}_{\mathcal{E}}, \quad B(T_\phi u) = B(u),$$
	$$ \forall \theta \in \R, \forall u \in L^{2,1}_{\mathcal{E}}, \quad B(L_\theta u) = B(u),$$
	$$ \forall \alpha \in \C, \forall u \in L^{2,1}_{\mathcal{E}}, \quad B(R_\alpha u)= B(u).$$
\end{lemme}
\begin{proof}
	The first assertion is obvious by invariance of the Hamiltonian, the mass, the angular momentum and the magnetic momentum by phase rotations. The second one is obtained by using~\eqref{action_rot_Q}, that gives $|Q(L_\theta u)| = |Q(u)|$, and by using the invariance laws of the other quantities under $L_\theta$. For the last assertion we use both Eqs~\eqref{action_trans_P} and~\eqref{action_trans_Q}, just as the invariance of~$\mathcal{H}$ and~$M$ under magnetic translation:
\begin{align*}
    B(R_\alpha u) &= 4\pi\mathcal{H}(R_\alpha u)+\frac{1}{4}   \pa{M(R_\alpha u)P(R_\alpha u)-|Q(R_\alpha u)|^2}-\frac12 M^2(R_\alpha u) \\
    &= 4\pi \mathcal{H}(u)+\frac{1}{4} \pa{M(u)(P(u) - 2\Re(\overline{\alpha}Q(u))+|\alpha|^2M(u))-|Q(u) - \alpha M(u)|^2}-\frac12 M^2(u)\\
    &= B(u),
\end{align*}
after developing $|Q(u) - \alpha M(u)|^2 = |Q(u)|^2 -2\Re(\overline{\alpha}Q(u)) + |\alpha|^2M(u)^2.$
\end{proof}

Define $E$ by
\begin{align*}
    E(u) &= 4\pi \mathcal{H}(u)+\frac{1}{4} M(u)P(u)-\frac12 M^2(u)\\
    &=\frac12 \sum_{j=0}^{+\infty} \frac{j!}{2^j} \abs{\sum_{k=0}^{j}\frac{a_k a_{j-k}}{\sqrt{k ! (j-k)!}}}^2+\frac{1}{4}  \sum^{+\infty}_{n,k=0}(n-2) \ov{a_n}\ov{a_k}{a}_n{a}_k.
\end{align*}
Therefore, $$B(u) = E(u) - \frac14|Q(u)|^2.$$

Observe that 
\begin{prop}\label{prop_BE}
    $B\geq 0 $ is equivalent to $E\geq 0$.
\end{prop}
\begin{proof}
    We have $E \geq B$ and if $E(u)\geq 0 $ for all $u\in \E$, then for all $\alpha \in \C$
    \begin{equation*}
        E(R_\alpha u)= 4\pi \mathcal H(u)+\frac{1}{4}\pa{ M(u)P(u) - 2M(u) \Re(\ov \alpha Q(u)) + |\alpha|^2 M^2(u)}-\frac12 M^2(u) \geq 0,
    \end{equation*}
    whose minimum is obtained for $\alpha= Q(u)/M(u)$. With this particular choice of~$\alpha$, we recover~$B(u)$.
\end{proof}

\begin{remark}
    In order to prove Theorem \ref{thm1}, we will study the functional $B$ instead of $E$. The method we use does not work for $E$, see Remark \ref{RemBE} below.
\end{remark}

\begin{prop}\label{prop_Fmu}
    Assume that for all $u \in L^{2,1}_{\E}$, $B(u)\geq 0$. Then Theorem \ref{thm1} $(ii)$ holds true.
\end{prop}
\begin{proof}
	For any~$\mu>0,$ the function $\phi_0$ is a global minimizer of \eqref{min_pb} if and only if, for all function~$u$ such that~$M(u)=1$, $\mathcal{G}_\mu(u) \geq \mathcal{G}_\mu(\phi_0).$ As suggested in~\cite[Proposition 7.5]{GGT}, this is equivalent by an homogeneity argument to~$F_\mu \geq 0$, where:
\begin{equation*}
    F_\mu(u)=8\pi \mathcal H(u)+M(u)(\mu P(u)-M(u)).
\end{equation*}
We check that 
\begin{equation*}
    F_\mu(u) = 2E(u)+(\mu-\frac12) P(u)M(u).
\end{equation*}
Suppose that for all $u \in L^{2,1}_{\E}$, $B(u)\geq 0$.  Then, Proposition \ref{prop_BE} implies that $$\forall u \in L^{2,1}_{\E}, \quad  F_{\frac12}(u) = 2E(u) \geq 0,$$
meaning that $\phi_0$ is a global minimizer for $\mu = \frac12.$ It follows, for $\mu_1 > \frac12$ and $u \in L^{2,1}_{\E},$
$$ 1 = \mathcal{G}_{\mu_1}(\phi_0) = \mathcal{G}_{\frac12}(\phi_0) \leq \mathcal{G}_{\frac12}(u) \leq \mathcal{G}_{\mu_1}(u),$$
so that $\phi_0$ is also a global minimizer for $\mu = \mu_1$. Furthermore, 
$$  F_{\mu_1}(u) = 2E(u)+(\mu_1-\frac12) P(u)M(u) \geq (\mu_1-\frac12) P(u)M(u).$$
We deduce $F_{\mu_1}(u) > 0$ for $u \neq c \phi_0$ due to \eqref{MP_sum}, and Theorem \ref{thm1} $(ii)$ holds true.
\end{proof}

Set $A_{jk}=\frac{a_k a_{j-k}}{\sqrt{k ! (j-k)!}}$ for $0\leq k \leq j$. We use the convention $A_{j,j+1}=0$ and  $A_{j,-1}=0$. Then 

\begin{lemme} The functional $B$ is a quadratic form in the $(A_{jk})$, and the following formula holds true: 
\begin{eqnarray}\label{exp_B}
    B(u) &=& \frac12 \sum^{+\infty}_{j=0} \sum^{j}_{k=0}  \ov{A_{jk}} \pac{\frac{j!}{2^j} \sum_{\ell=0}^{j} A_{j\ell}  +\frac{k! (j-k)!}{4} \pa{(j-4) A_{jk}-(k+1)A_{j,k+1}   -(j-k+1)A_{j,k-1} }}  \nonumber \\
    &=&\sum^{+\infty}_{j=0} \sum^{j}_{k,\ell=0} B_{k\ell}^{(j)} \ov{A_{jk}}{A_{j\ell}}.
\end{eqnarray}
\end{lemme}
This expression of $B$ appears first in \cite[equation (24)]{Clerck-Evnin}.

\begin{proof} The proof is elementary. We have
    \begin{equation}\label{eq1}
        4\pi \mathcal{H}(u)=\frac12 \sum_{j=0}^{+\infty} \frac{j!}{2^j} \abs{\sum_{k=0}^{j}A_{jk}}^2   
    \end{equation}

    \begin{eqnarray}\label{eq2}
        \sum_{n,k=0}^{+\infty}(n-2) \ov{a_n}\ov{a_k} {a}_n {a}_k &=& \frac12  \sum^{+\infty}_{n,k=0}(n+k-4) \ov{a_n} \ov{a_k}{a}_n{a}_k \nonumber \\
        &=& \frac12\sum_{j=0}^{+\infty}  \sum^{j}_{k=0}(j-4) k! (j-k)! \abs{\frac{{a_k}{a_{j-k}}}{\sqrt{k!(j-k)!}} }^2 \nonumber \\
        &=& \frac12\sum_{j=0}^{+\infty}  \sum^{j}_{k=0}(j-4) k! (j-k)! |A_{jk}|^2
    \end{eqnarray}
    \begin{eqnarray}\label{eq3}
        \sum^{+\infty}_{n,k=0}\sqrt{n+1}\sqrt{k+1}  \ov{a_{n+1}} \ov{a_k}  {a}_n{a}_{k+1} &=& \sum^{+\infty}_{j=1}\sum_{k=0}^{j-1} \sqrt{j-k}\sqrt{k+1}  \ov{a_{j-k}} \ov{a_k}  {a}_{j-(k+1)}{a}_{k+1} \nonumber \\
        &=& \sum^{+\infty}_{j=0} \sum^{j}_{k=0}  k! (j-k)!    (k+1) \ov{A_{jk}}  A_{j,k+1}.   
    \end{eqnarray}
    In the latter formula, the change of variables $k \mapsto j-k$ gives
    \begin{equation}\label{eq4}
    \sum^{+\infty}_{n,k=0}\sqrt{n+1}\sqrt{k+1}\ov{a_{n+1}} \ov{a_k}{a}_n{a}_{k+1}
    = \sum_{j=0}^{+\infty} \sum_{k=0}^{j}k!(j-k)!(j-k+1)\ov{A_{jk}} A_{j,k-1}.       
    \end{equation}
    Therefore, by \eqref{eq1}, \eqref{eq2}, \eqref{eq3} and \eqref{eq4}, we obtain the formula.
\end{proof}

\subsection{Digression on centrosymmetric matrices}\label{sec_centro}
In this subsection, we study a class of matrices called centrosymmetric matrices, in which the family $(B^{(j)})$ lies. We focus on the results \cite[Lemma 2, Lemma 3, Theorem 2]{Cantoni-Butler}, giving reduction properties for centrosymmetric matrices, and we refer to \cite{Cantoni-Butler} for more details, and for proofs. Let $n \geq 2$ be a fixed integer.
\begin{df}
	We define $J$ the symmetric matrix $\text{antidiag}(1,n)=(J_{i,j})_{1\leq i,j\leq n} \in \mathcal{M}_n(\R)$, that is to say the matrix which is defined by
	$$\left\{
        \begin{array}{ll}
        J_{i,j} = 1 & \; \text{if} \; j=n+1-i, \\
        J_{i,j} = 0 & \; \text{otherwise.}
        \end{array}
    \right.$$
\end{df}
\begin{prop}
    Let $A$ be a matrix of order~$n$. Then the matrix~$JA$ is the upside-down matrix of~$A$, i.e. the matrix obtained from~$A$ by horizontal symmetry with the middle of~$A$, whereas the matrix~$AJ$ is the matrix obtained from~$A$ by vertical symmetry.
\end{prop}
\begin{df}\label{df_centrosym}
	A vector $v$ of order $n$ is symmetric if it satisfies $$ Jv = v,$$
	which reads $$v_i = v_{n+1-i}, \quad 1 \leq i \leq n.$$
	A vector $v$ of order $n$ is skew symmetric if it satisfies $$ Jv = -v,$$
	which reads $$v_i = -v_{n+1-i}, \quad 1 \leq i \leq n.$$
	A matrix $A$ of order $n$ is centrosymmetric if it satisfies $$ JAJ = A,$$
	which reads $$A_{i,j} = A_{n+1-i,n+1-j}, \quad 1 \leq i,j \leq n.$$
\end{df}
We denote by $\mathcal{C}_n$ the set of both symmetric and centrosymmetric real matrices of order~$n$. By Definition \ref{df_centrosym} it directly follows:
\begin{lemme}\label{lem3.12}
    Any matrix which is centrosymmetric preserves the space of the symmetric vectors and the space of the skew symmetric vectors.
\end{lemme}
\begin{proof}
    Let $A$ be a centrosymmetric matrix, and~$v$ a symmetric vector (the same proof applies for a skew symmetric vector). Since $v$ is symmetric, we have $Av = AJv$, so $JAv = JAJv = Av$ because $A$ is centrosymmetric, meaning $Av$ is still a symmetric vector.
\end{proof}
Moreover, we have the following lemmas (see~\cite[Lemma~2, Lemma~3, Theorem~2]{Cantoni-Butler}):
\begin{lemme}\label{lemme_CB_pair}
	Suppose that the integer $n \geq 2$ is even. Let $X\in\mathcal{C}_n$. Then $X$ takes the form
	$$X=\begin{pmatrix}
	A & ^tC  \\[2pt]
	C & JAJ                 
	\end{pmatrix}, $$
	with $A,C \in \mathcal{M}_{n/2}(\R)$ two matrices that satisfy $^tA=A$ and $^tC = JCJ.$ Furthermore, $X$ is orthogonally similar to the matrix
	$$\begin{pmatrix}
	A-JC & 0 \\[2pt]
	0 & A+JC
	\end{pmatrix},$$
	and the eigenvalues of $X$ are the eigenvalues of both $A+JC$ and $A-JC$. Note that the eigenvectors of $X$ corresponding to the eigenvalues of $A+JC$ are symmetric, whereas those corresponding to the eigenvalues of $A-JC$ are skew symmetric.
\end{lemme}
\begin{lemme}\label{lemme_CB_impair}
	Suppose that the integer $n \geq 1$ is odd. Let $X\in\mathcal{C}_n$. Then $X$ takes the form
	$$X=\begin{pmatrix}
	A & x & ^tC  \\[2pt]
	^tx & q & ^txJ  \\[2pt]
	C & Jx & JAJ          
	\end{pmatrix}, $$
	with $A,C \in \mathcal{M}_{m}(\R)$, with $m=\frac{n-1}{2}$ two matrices that satisfy $^tA=A$ and $^tC = JCJ,$ and where $q$ is a real number and $x$ a real vector of order $m=\frac{n-1}{2}.$ Furthermore, $X$ is orthogonally similar to the matrix
	$$\begin{pmatrix}
	A-JC & 0 & 0  \\[2pt]
	0 & q & \sqrt{2}\,\!^tx  \\[2pt]
	0 & \sqrt{2}\,\!x & A+JC         
	\end{pmatrix},$$
	and the eigenvalues of $X$ are the eigenvalues of both the matrices $$Y=\begin{pmatrix} A+JC &  \sqrt{2}\,\!x\\[2pt] \sqrt{2}\,\!^tx & q \end{pmatrix} \quad \text{and} \quad A-JC.$$
	Note that the eigenvectors of $X$ corresponding to the eigenvalues of $Y$ are symmetric, whereas those corresponding to the eigenvalues of $A-JC$ are skew symmetric.
\end{lemme}

\subsection{Study of the matrices \texorpdfstring{$(B^{(j)})$}{}}
\begin{prop}
    For all $j \geq 1,$ the real matrix $B^{(j)}\in \mathcal{M}_{j+1}(\R)$ is both symmetric and centrosymmetric.
\end{prop}
\begin{proof}
    It is enough to show that $$B_{k,\ell}^{(j)} = B_{\ell,k}^{(j)} = B_{j-k,j-\ell}^{(j)}, \qquad 0 \leq k,\ell \leq j,$$ which is easy by the expression \eqref{exp_B}.
\end{proof}

In particular, by Lemma \ref{lem3.12}, $B^{(j)}$ preserves the space $\mathcal{F}^{(j)}$ of the symmetric vectors of size $j+1$ and the space $\mathcal{G}^{(j)}$ of the skew symmetric vectors of size $j+1$.

\begin{prop}\label{Prop_B0}
	For all $u \in L^{2,1}_{\E}$, $B(u)\geq 0$.
\end{prop}
\begin{proof}
	Since the $(A_{j\ell})$ are symmetric, it is enough to show that, for all $j \in \N$, $B^{(j)}_{| \mathcal{F}^{(j)}}$ is non-negative. We split the proof into two parts. The case $j$ odd is obtained in Lemma \ref{lemme_cas_impair} below, and the case $j$ even is obtained in Lemma \ref{lemme_cas_pair} below. The result follows.
\end{proof}

\begin{remark}\label{RemBE}
	Even though we can write $E$ in the same way with centrosymmetric matrices, the same argument for $E$ instead of $B$ does not hold true: for $j=3$, the corresponding matrix has a negative eigenvalue. Actually, the coefficient in $B$ are better arranged than those in $E$, because the action of $R_\alpha$ leaves $B$ invariant, whereas $E$ is not left invariant.
\end{remark}

\begin{lemme}\label{lemme_cas_impair}
	For all $j \geq 1$ odd, we have $B^{(j)}_{| \mathcal{F}^{(j)}}$ non-negative.
\end{lemme}
\begin{proof}
We fix here an odd integer $j$ such that $j \geq 1$.

{\underline    {Step 1 : Reduction to a symmetric matrix of lower order.}} Denote by $J= \text{antidiag}(1,\frac{j+1}{2})$. Here we follow \cite{Cantoni-Butler}, and more precisely we use Sect.~\ref{sec_centro} above. The matrix $B^{(j)} \in \mathcal{M}_{j+1}(\R)$ can be written 
 $$B^{(j)}  =   \begin{pmatrix}
    A & JCJ  \\[2pt]
    C & JAJ                 
\end{pmatrix}.$$
By Lemma \ref{lemme_CB_pair}, the eigenvalues of $B^{(j)}_{| \mathcal{F}^{(j)}}$ are the eigenvalues of $S^{(j)}:=A+JC$. We shall then study the matrix $S^{(j)}$ instead of $B^{(j)}$. A quick computation gives 
$$B^{(1)}  = \frac18 \begin{pmatrix}
-1 & 1  \\[2pt]
1 &  -1                 
\end{pmatrix}, \qquad S^{(1)}=0,$$
\begin{equation*}
    B^{(3)}  = \frac18 \begin{pmatrix}
    -3 & -3 & 3&3\\[2pt]
    -3 & 1 & -1&3\\[2pt]
    3 & -1 & 1&-3\\[2pt]
    3 & 3 & -3&-3     
\end{pmatrix}, \qquad S^{(3)}=0,
\end{equation*}
and
\begin{equation*}
    B^{(5)}  = \frac38 \begin{pmatrix}
    45 & -35 & 5 & 5 & 5 & 5 \\[2pt]
    -35 & 13 & -11 & 5 & 5 & 5 \\[2pt]
    5 & -11 & 9 & -7 & 5 &  5\\[2pt]
    5 & 5 & -7 & 9 & -11 & 5 \\[2pt]
    5 & 5 & 5 & -11 & 13 & -35 \\[2pt]
    5 & 5 & 5 & 5 & -35 & 45
\end{pmatrix}, \qquad 
S^{(5)}= \frac34
\begin{pmatrix} 25 & -15 & 5\\[2pt]
    -15 & 9 & -3 \\[2pt]
    5 & -3 & 1 
\end{pmatrix}.
\end{equation*}
We can check that $S^{(5)} \geq 0,$ with $0$ being an eigenvalue with multiplicity $2$. Then, we can focus on the case $j \geq 7$ odd. In this case, we can check that $0$ is an eigenvalue of $S^{(j)}$ of multiplicity at least~$2$, and the eigenvectors associated are given by $v^{(j)}$ and $w^{(j)}$, where:
\begin{equation}\label{vecteurs_propres_B_impair}
    v_k^{(j)} = \frac{1}{(k-1)!(j-k+1)!}, \quad w_k^{(j)} = \frac{(k-1)(j-k+1)}{(k-1)!(j-k+1)!}; \qquad 1\leq k \leq \frac{j+1}{2}.
\end{equation}

If we write $B^{(j)} = (B_{k,\ell}^{(j)})_{1\leq k,\ell\leq j+1}$, we have 
$$B_{k,k}^{(j)} = \frac{j!}{2^{j+1}} + \frac{1}{8}(j-4)(k-1)!(j-k+1)!, \qquad 1\leq k \leq j+1,$$
$$B_{k,k+1}^{(j)} = B_{k+1,k}^{(j)} = \frac{j!}{2^{j+1}} - \frac{1}{8}k!(j-k+1)!, \qquad 1\leq k \leq j,$$
and $B_{k,\ell}^{(j)} =\dfrac{j!}{2^{j+1}}$ otherwise.

We consider $p = \frac{j+1}{2},$ and then write the matrix $S^{(j)}:=A+JC$ as 
$$ S^{(j)}:=T^{(j)} + K^{(j)} \in \mathcal{M}_{\frac{j+1}{2}}(\R) = \mathcal{M}_p(\R),$$
where $K^{(j)} = \frac{j!}{2^j}Ones(\frac{j+1}{2},\frac{j+1}{2}) = \frac{j!}{2^j}Ones\pa{p,p}$, and $T^{(j)}\in\mathcal{M}_p(\R)$ is a tridiagonal and symmetric matrix. More precisely, we have 
\begin{align*}
    &T^{(j)}_{k,k} = \frac{1}{8}(2p-5)(k-1)!(2p-k)!  & 1 \leq k \leq p-1,\\
    &T^{(j)}_{p,p} = \frac{1}{8}(p-5)(p-1)!p!, &\\
    &T^{(j)}_{k,k+1} = T^{(j)}_{k+1,k} = -\frac{1}{8}k!(2p-k)!& 1 \leq k \leq p-1.
\end{align*}

	{\underline    {Step 2 : Reduction to the study of a tridiagonal matrix.}}
For $m \geq 1$ an integer, we denote by~$\mathcal{O}_m(\R)$ the subset of $\mathcal{M}_m(\R)$ constituted of orthogonal matrices.
The matrix~$K^{(j)}$ defined above is a real symmetric matrix of rank one, with eigenvalue~$0$ of multiplicity~$p-1$. The only remaining eigenvalue of~$K$ is its trace $$\text{trace}(K^{(j)}) = p\frac{j!}{2^{j}} = \frac{(2p)!}{2^{2p}}=:\delta,$$ which is of multiplicity~$1$. Here we follow \cite[pages 94 to 98]{Wilkinson}. We give proofs of the lemmas for the sake of completeness. There exists, by the spectral theorem, an orthogonal matrix $R \in \mathcal{O}_p(\R)$ such that 
$$ ^tRK^{(j)}R = \text{diag}(\delta, 0 , \dots, 0).$$
Then we write 
$$^tR T^{(j)}R = \begin{pmatrix} \alpha & ^ta \\ a & T_{p-1} \end{pmatrix},$$
with $T_{p-1}$ a symmetric matrix of size $p-1$, $\alpha \in \R$ a scalar, and $a$ a vector of size $p-1$. Since $T_{p-1}$ is real and symmetric, we can write
$$ ^tUT_{p-1}U = \text{diag}(\alpha_1, \dots, \alpha_{p-1}), \qquad U \in \mathcal{O}_{p-1}(\R), $$
and where the $\alpha_i$ are the eigenvalues of $T_{p-1}$ counted with multiplicity. Then, we consider the matrix $$Q=R\begin{pmatrix} 1 & 0\\ 0 & U \end{pmatrix}.$$
We have $Q\in \mathcal{O}_p(\R)$ and we compute
\begin{align*}
    ^tQS^{(j)}Q &= \,^tQ(K^{(j)}+T^{(j)})Q = \,^tQK^{(j)}Q+  \,^tQT^{(j)}Q \\
    &= \begin{pmatrix} 1 & 0\\ 0 & ^tU \end{pmatrix}\, ^tRK^{(j)}R\begin{pmatrix} 1 & 0\\ 0 & U \end{pmatrix}+  \begin{pmatrix} 1 & 0\\ 0 & ^tU \end{pmatrix} \,^tRT^{(j)}R\begin{pmatrix} 1 & 0\\ 0 & U \end{pmatrix}\\
    &=\begin{pmatrix} 1 & 0\\ 0 & ^tU \end{pmatrix}\, \begin{pmatrix} \delta & 0\\0 & 0 \end{pmatrix}\begin{pmatrix} 1 & 0\\ 0 & U \end{pmatrix}+  \begin{pmatrix} 1 & 0\\ 0 & ^tU \end{pmatrix} \,\begin{pmatrix} \alpha & ^ta \\ a & T_{p-1} \end{pmatrix}\begin{pmatrix} 1 & 0\\ 0 & U \end{pmatrix} \\
    &=\begin{pmatrix} \delta & 0\\0 & 0 \end{pmatrix} + \begin{pmatrix} \alpha & ^taU \\ ^tUa & ^tUT_{p-1}U \end{pmatrix},
\end{align*}

where $\delta = \frac{(2p)!}{2^{2p}}.$ That is to say 
\begin{equation*}
    ^tQS^{(j)}Q = \underbrace{\begin{pmatrix} \alpha & b_1 & b_2 & \cdots & b_{p-1} \\
    b_1 & \alpha_1 & 0 & & \\
    b_2 & 0 & \alpha_2 & 0  & \\
    \vdots &  & 0 & \ddots  & 0 \\
    b_{p-1} &  &  & 0  & \alpha_{p-1}
    \end{pmatrix}}_{= ^tQT^{(j)}Q} 
    + \begin{pmatrix}
    \delta & 0 & 0 & \cdots & 0 \\
    0 & 0 & 0 & & \\
    0 & 0 & 0 & 0  & \\
    \vdots &  & 0 & \ddots  & 0 \\
    0 &  &  & 0  & 0
    \end{pmatrix}, \qquad b={^tU}a.
\end{equation*}

\begin{lemme}\label{lemma_eig1}\cite[pp. 95-96]{Wilkinson}
    We consider the bordered diagonal matrix 
    \begin{equation*}
        M := \begin{pmatrix} \alpha & b_1 & b_2 & \cdots & b_{p-1} \\  b_1 & \alpha_1 & 0 &   & \\  b_2 & 0 & \alpha_2 & 0  & \\  \vdots &  & 0 & \ddots  & 0 \\  b_{p-1} &  &  & 0  & \alpha_{p-1}  \end{pmatrix}.
    \end{equation*}
    Then:
    \begin{enumerate}
        \item If all the $\alpha_i$ are distinct, and for all $i, b_i\neq 0$, then the eigenvalues of $M$ are given by the solutions of the equation of unknown $\lambda \in \R$ 
        \begin{equation}\label{eig_bord_diag}
            \alpha = \lambda + \sum_{i=1}^{p-1} \frac{b_i^2}{\alpha_i-\lambda}.
        \end{equation}
        \item If some of the $b_i$ are null, then $\alpha_i$ is a eigenvalue of~$M$. The other eigenvalues are given by~\eqref{eig_bord_diag}.
        \item If $r \geq 2$ is an integer, and $i_1 < i_2 < \dots < i_{r}$ some integers satisfying $\alpha_{i_1}=\alpha_{i_2}= \cdots =\alpha_{i_r},$ then~$\alpha_{i_1}$ is an eigenvalue of~$M$ with multiplicity~$r-1$.
    \end{enumerate}
\end{lemme}

\begin{proof}
    If there exists $i$ such that $b_i=0$, it is clear that $\alpha_i$ is an eigenvalue of $M$, and we can extract from $M$ the sub-matrix $M'$ where we deleted the $(i+1)$-th line and column of $M$. The remaining eigenvalues of $M$ are exactly the eigenvalues of $M'$. We can repeat this process for all $j$ such that $b_j =0$. Now we treat the case where all the $b_i$ are non zeros. We compute the characteristic polynomial of $M$:
    \begin{equation}\label{charpoly}
    \rchi_{M}(X) = \text{det}(XI_p - M) = (X-\alpha)\prod_{i=1}^{p-1}(X-\alpha_i) - \sum_{k=1}^{p-1}b_k^2\prod_{i\neq k}(X-\alpha_i).
    \end{equation}
    We then divide by $\prod_{i=1}^{p-1}(\lambda-\alpha_i)$ the equation $\rchi_{M}(\lambda)=0$ to obtain 
    \begin{equation*}
        \alpha = \lambda - \sum_{k=1}^{p-1}\frac{b_k^2}{\lambda - \alpha_k},
    \end{equation*}
    which gives the first part of the result. \\
    Let's treat the case were the $\alpha_i$ are non distinct. If so, we write $t$ the number of distinct $\alpha_i$, and we denote by $r_1,\dots,r_t$ their multiplicities, with $p-1 = r_1+\cdots +r_t$. We notice in \eqref{charpoly} that $\rchi_{M}(X)$ can be factorised by $$\prod_{i=1}^{t}(X-\alpha_i)^{r_i-1}.$$
    We deduce that $\alpha_i$ is an eigenvalue of $M$ of multiplicity $(r_i-1).$
\end{proof}

\begin{lemme}\label{lemma_eig2}
	We denote by $\lambda_1 \leq \lambda_2 \leq \dots \leq \lambda_p$ and $\lambda_1' \leq \lambda_2' \leq \dots \leq \lambda_p'$ respectively the eigenvalues of $T^{(j)}$ and $S^{(j)}$. Then, there exists a family $(m_i)_{1\leq i \leq p}$ of real number such that 
\begin{equation*}
    \left\{
        \begin{array}{ll}
        \lambda_i' - \lambda_i = m_i\delta, \\
        0 \leq m_i \leq 1,\\
        \sum_{i=1}^p m_i = 1.
        \end{array}
    \right.
\end{equation*}
Furthermore, we have 
\begin{equation}\label{entrelacement_vp}
    \lambda_1 \leq \lambda_1' \leq \lambda_2 \leq \lambda_2' \leq \cdots \leq \lambda_p \leq\lambda_p'.
\end{equation}
\end{lemme}

\begin{proof}
If the $\alpha_i$ are not necessarily distinct, then both $T^{(j)}$ and $S^{(j)}$ have $\alpha_i$ as eigenvalue (say $\lambda_i$), and we can take $m_i=0$. Let's turn to the case when all the $\alpha_i$ are distinct.
We apply Lemma \ref{lemma_eig1} to the matrices $^tQS^{(j)}Q$ and $^tQT^{(j)}Q$, whose eigenvalues are those of $S^{(j)}$ and $T^{(j)}$ since $Q \in \mathcal{O}_p(\R)$. We obtain that the eigenvalues of $T^{(j)}$ solve the equation 
\begin{equation*}
    \alpha = \lambda + \sum_{k=1}^{p-1}\frac{b_k^2}{\alpha_k-\lambda},
\end{equation*}
and those of $S^{(j)}$ solve the equation
\begin{equation*}
    \alpha + \delta = \lambda + \sum_{k=1}^{p-1}\frac{b_k^2}{\alpha_k-\lambda}.
\end{equation*}
We then consider the function $$f: \lambda \longmapsto \lambda + \sum_{k=1}^{p-1}\frac{b_k^2}{\alpha_k-\lambda},$$
so that the previous equations can be written $\alpha = f(\lambda)$ and $\alpha+\delta = f(\lambda)$ respectively. We denote by $\lambda_1 \leq \lambda_2 \leq \dots \leq \lambda_p$ and $\lambda_1' \leq \lambda_2' \leq \dots \leq \lambda_p'$ respectively the eigenvalues of $T^{(j)}$ and $S^{(j)}$.
We compute 
$$\frac{df}{d\lambda} = 1 + \sum_{k=1}^{p-1}\frac{b_k^2}{(\alpha_k-\lambda)^2} > 1,$$
which means, since $(\alpha+\delta) - \alpha= \delta >0$, that 
$$0 \leq \lambda_i'-\lambda_i \leq \delta.$$

Note that we do not always have $0 < \lambda_i'-\lambda_i$ because $b_{i-1} =0$ implies $\lambda_i=\lambda_i'(=\alpha_{i-1})$ (see above), and we do not have  $\lambda_i'-\lambda_i < \delta$ when all the $b_i$ equal $0$. Then there exists $m_i$ a real number such that $0 \leq m_i \leq 1$ and $\lambda_i'-\lambda_i =m_i \delta.$
Furthermore, the sum of the eigenvalues of a matrix equals its trace, so we have 
$$ \delta \sum_{i=1}^{p}m_i = \sum_{i=1}^{p}(\lambda_i'-\lambda_i) = \text{trace}(S^{(j)}) - \text{trace}(T^{(j)}) = \pa{\alpha+\delta+\sum_{i=1}^{p-1}\alpha_i} - \pa{\alpha+\sum_{i=1}^{p-1}\alpha_i} = \delta,$$
and we deduce that $\sum_{i=1}^{p}m_i=1.$ Finally, 
\begin{equation*}
    \left\{
        \begin{array}{ll}
        \lambda_i' - \lambda_i = m_i\delta, \\
        0 \leq m_i \leq 1,\\
        \sum_{i=1}^p m_i = 1.
        \end{array}
    \right.
\end{equation*}
The inequalities \eqref{entrelacement_vp} come from the fact that there is exactly one solution of $f(\lambda) = \beta$ on each interval whose boundaries are two consecutive $\alpha_i$, for any $\beta \in \R$.
\end{proof}

\begin{prop}\label{prop_eig1}
    If we have $\lambda_3 > 0$, then $S^{(j)} \geq 0$.
\end{prop}
\begin{proof}
    If the third eigenvalue (ordered increasingly) of the matrix $T^{(j)}$ is positive, then $\lambda_i' \geq \lambda_i \geq \lambda_3 > 0, \forall i \geq 3$. Since $0$ is an eigenvalue of $S^{(j)}$ with multiplicity at least $2$, we have necessarily $\lambda_1' = \lambda_2' = 0$, and then $S^{(j)} \geq 0$.
\end{proof}


	{\underline    {Step 3 : Study of a Sturm sequence.}} This step follows the pages $300$ to $302$ of \cite{Wilkinson}. For convenience, we write:
$$T^{(j)} = \begin{pmatrix} a_1 & b_{1,2} & 0 & & \\  b_{1,2} & a_2 & b_{2,3} &   & \\  0 & b_{2,3} & a_3 & \ddots  & \\  &  & \ddots & \ddots  & b_{p-1,p} \\   &  &  & b_{p-1,p}  & a_{p}  \end{pmatrix}.$$
We denote by $\Delta_r(\lambda)$ the principal minor of order $r$ of the matrix $T^{(j)}-\lambda I_p$, which is $\det(A_r(\lambda))$, where: 
$$A_r(\lambda) = \begin{pmatrix} a_1-\lambda & b_{1,2} & 0 & & \\  b_{1,2} & a_2-\lambda & b_{2,3} &   & \\  0 & b_{2,3} & a_3 -\lambda & \ddots  & \\  &  & \ddots & \ddots  & b_{r-1,r} \\   &  &  & b_{r-1,r}  & a_{r}-\lambda  \end{pmatrix},$$
for $r \ge 1$. We also impose $\Delta_0(\lambda) = 1$. We then have $ \Delta_1(\lambda) = a_1 - \lambda$ and for all $2 \leq i \leq p$: 
	\begin{equation}\label{relat_deltak}
	\Delta_i(\lambda) = (a_i-\lambda) \Delta_{i-1}(\lambda) - b_{i-1,i}^2\Delta_{i-2}(\lambda).
	\end{equation}
We have the following lemma (see \cite[pages 300-302]{Wilkinson}): 

\begin{lemme}\label{lemme_mineurs}
	Let the $\Delta_i(\mu)$ be the principal minors evaluated in a certain real number $\mu$. If $s(\mu)$ stands for the number of correspondence of sign in the consecutive terms of the sequence $\Delta_0(\mu), \Delta_1(\mu), \dots, \Delta_p(\mu)$, then $s(\mu)$ is exactly the number of eigenvalues of $T^{(j)} = A_p$ which are strictly above $\mu$. If one of the terms of this sequence is null, say $\Delta_r(\mu) = 0$, we consider that it has the opposite sign of $\Delta_{r-1}(\mu).$
\end{lemme}
\begin{remark}
	If all the $b_{i-1,i}$ are non null, then we can't have two consecutive $\Delta_i(\mu)$ that are null (by decreasing induction we would have all the other $\Delta_i(\mu)=0$, and in particular $\Delta_0(\mu)=0$, which is absurd).
\end{remark}

By Proposition \ref{prop_eig1}, it is sufficient to show that at most two eigenvalues of~$T^{(j)}$ are non-positive. In order to prove that, we extract the first~$(p-1)$ lines and columns of~$T^{(j)}$ and denote by~$T$ the matrix obtained by this process. By strict separation of the eigenvalues of the extracted matrix, it is sufficient to show that at most one of the eigenvalues of~$T$ is non-positive. We use Lemma~\ref{lemme_mineurs} with $\mu = 0$. If we are able to prove that there is at most one change in the signs of the sequence $\Delta_0(0), \Delta_1(0), \dots, \Delta_{p-1}(0)$, we would have the result of positivity of the initial matrix~$S^{(j)}$ (for $j$ odd). We now show:
\begin{lemme}
	We can write, for all $k \leq p-1$: 
	$$\Delta_k(0) = \gamma_{p,k} u_k,$$
	where the $u_k$ are integers satisfying the recurrence relation of order $2$
	\begin{equation}\label{relat_uk}
	u_{k+2} = (2p-5)u_{k+1}-(k+1)(2p-(k+1))u_{k},\; k \geq 0, \qquad u_0=1, \;u_1=2p-5,
	\end{equation}
	and 
	\begin{equation}\label{def_gammapk}
	\gamma_{p,k} =\frac{1}{8^k} 0! 1! \dots (k-1)! \times (2p-1)! \dots (2p-k)!.
	\end{equation}
\end{lemme}
\begin{proof}
	We first recall that 
	\begin{align*}
    	& a_k = \frac{1}{8}(2p-5)(k-1)!(2p-k)!  & 1 \leq k \leq p-1,\\
    	&b_{k,k+1} = -\frac{1}{8}k!(2p-k)!& 1 \leq k \leq p-2.
	\end{align*}
	We have $\Delta_0(0) = 1 = \gamma_{p,0}u_0$ and $ \Delta_1(0) = a_1 = \frac{1}{8}(2p-5)0!(2p-1)! = \gamma_{p,1}u_1.$ We proceed by induction. Suppose that we can write $\Delta_j(0) = \gamma_{p,j} u_j$ for all $j$ up to a certain $k$ such that $p-2 \geq k \geq 1$, with the relations \eqref{relat_uk} and \eqref{def_gammapk}. Then we have, by \eqref{relat_deltak}: 
	\begin{align*}
		\Delta_{k+1}(0) &= a_{k+1} \Delta_k(0) - b_{k,k+1}^2 \Delta_{k-1}(0) \\
		&= \frac{1}{8}(2p-5)k!(2p-1-k)! \gamma_{p,k} u_k  -\frac{1}{8^2}(k!(2p-k)!)^2\gamma_{p,k-1} u_{k-1} \\
		&= \frac{1}{8}k!(2p-1-k)!\pa{\gamma_{p,k} (2p-5)u_k - \frac{1}{8}(2p-k)(2p-k)!k!\gamma_{p,k-1}u_{k-1}} \\
		&= \frac{1}{8}k!(2p-1-k)!\gamma_{p,k}\pa{(2p-5)u_k - k(2p-k)u_{k-1}} = \gamma_{p,k+1}u_{k+1},
	\end{align*}
	where the last two equalities come from $\frac18(j-1)!(2p-j)!\gamma_{p,j-1} = \gamma_{p,j}.$
\end{proof}

\begin{remark}
Note that $\Delta_k(0)$ and $u_k$ have same sign, so it is sufficient to study $u_k$. 
\end{remark}

\begin{lemme}\label{lemme_polynomial}
    Let $(y_n)_{n\in\N}$ be a sequence defined by the values $y_0,y_1\neq0$ and the relation $$y_{n+2} = ay_{n+1} - (n+1)(b-(n+1))y_n,$$
    with $a,b$ some integers, with $b \geq a$.\\
    Suppose we have $y_1=ay_0$. Then, the sequence $(y_n)$ is stationary to $0$ for $n$ big enough, if and only if $a$ and $b$ have opposite parities. More precisely, if we write 
	$$f(x) = \sum_{n=0}^{+\infty} \frac{y_{n}}{n!}x^{n},$$
	we get the expression  $$f(x) = y_0(1+x)^{\frac12(a+b-1)}(1-x)^{\frac12(b-a-1)}.$$
    In particular, $f$ is a polynomial if and only if $a$ and $b$ have opposite parities.
\end{lemme}
\begin{proof}
    We consider the function $$f(x) = \sum_{n=0}^{+\infty} \frac{y_{n}}{n!}x^{n},$$
    which is a power series. Suppose its radius of convergence $R$ is positive. Then, for all $x\in ]-R,R[$ we have: 
    \begin{align*}
        f'(x) &= \sum_{n=1}^{+\infty} \frac{y_{n}}{(n-1)!}x^{n-1} = y_1+ \sum_{n=0}^{+\infty} \frac{y_{n+2}}{(n+1)!}x^{n+1} \\ 
        &= y_1 + \sum_{n=0}^{+\infty} \frac{ ay_{n+1} - (n+1)(b-(n+1))y_n}{(n+1)!}x^{n+1} \\
        &= y_1 + a\sum_{n=0}^{+\infty} \frac{y_{n+1}}{(n+1)!}x^{n+1}-b\sum_{n=0}^{+\infty} \frac{y_{n}}{n!}x^{n+1}+\sum_{n=0}^{+\infty} \frac{(n+1)y_{n}}{n!}x^{n+1}.
    \end{align*}
    We check that:
    \begin{equation*}
        \sum_{n=0}^{+\infty} \frac{y_{n+1}}{(n+1)!}x^{n+1} = f(x) - y_0 \qquad \text{and} \qquad \sum_{n=0}^{+\infty} \frac{y_{n}}{n!}x^{n+1}=xf(x).
    \end{equation*}
    Now we compute: 
    \begin{equation*}
        \sum_{n=0}^{+\infty} \frac{(n+1)y_{n}}{n!}x^{n+1} = xf(x) + \sum_{n=1}^{+\infty} \frac{y_{n}}{(n-1)!}x^{n+1} = xf(x) +x^2\sum_{n=0}^{+\infty} \frac{y_{n+1}}{n!}x^{n} = xf(x)+x^2f'(x).
    \end{equation*}
    So we get 
    $$(1-x^2)f'(x) = \underbrace{y_1-ay_0}_{=0} +\pac{a-(b-1)x}f(x),$$
    which is solved for $$f(x) = \lambda(1+x)^{\frac12(a+b-1)}(1-x)^{\frac12(b-a-1)}, \qquad \lambda\in\R.$$
    With $f(0)=y_0 \neq 0$ we get $f(x) = y_0(1+x)^{\frac12(a+b-1)}(1-x)^{\frac12(b-a-1)},$ which is a polynomial if and only if $a$ and $b$ have opposite parities and satisfy $b > a$. Its degree is then $b-1$. In particular, we deduce that $y_n = 0$ for $n\geq b$. 
\end{proof}

We then use this latter lemma with $a=2p-5$ and $b=2p$. We obtain directly the expression 
$$f(x) = \sum_{n=0}^{+\infty} \frac{u_{n}}{n!}x^{n} = (1+x)^{2p-3}(1-x)^{2}.$$
In particular, $u_n=0$ for all $n$ such that $n \geq 2p$. We develop this last expression to obtain: 
    \begin{align*}
        f(x) &= (1-2x+x^2)\sum_{n=0}^{2p-3}\dbinom{2p-3}{n}x^n\\
        &= 1+(2p-5)x + \sum_{n=2}^{2p-3}\pac{\dbinom{2p-3}{n}-2\dbinom{2p-3}{n-1}+\dbinom{2p-3}{n-2}}x^n + (2p-5)x^{2p-2}+x^{2p-1}.
    \end{align*}
	Then, for $n$ such that $2 \leq n \leq 2p-3,$ we have 
    \begin{align*}
        u_n &= n! \pac{\dbinom{2p-3}{n}-2\dbinom{2p-3}{n-1}+\dbinom{2p-3}{n-2}} \\
        &= n!\pac{\frac{(2p-3)!}{n!(2p-3-n)!}-2\frac{(2p-3)!}{(n-1)!(2p-2-n)!}+\frac{(2p-3)!}{(n-2)!(2p-1-n)!}}  \\
        &= \frac{(2p-3)!}{(2p-1-n)!}\pac{(2p-2-n)(2p-1-n)-2n(2p-1-n)+n(n-1)} \\
        &= \frac{(2p-3)!}{(2p-1-n)!} \pa{4n^2-4(2p-1)n+(2p-1)(2p-2)}= 4\frac{(2p-3)!}{(2p-1-n)!}\pa{n-p_+}\pa{n-p_-},
    \end{align*}
    where $$p_+ = \frac{1}{2}\pa{2p-1+\sqrt{2p-1}}, \qquad \text{and} \qquad p_- = \frac{1}{2}\pa{2p-1-\sqrt{2p-1}}.$$
\begin{remark}
	We emphasise here that we have, for all $p \geq 2$ (which correspond to $j \geq3$), $p_- < p < p_+$.
\end{remark}

	{\underline    {Step 4 : Conclusion.}}
We recall that we have already dealt with the cases $j \in \{1,3,5\}$ in step~$1$, corresponding to $p\in \{1,2,3\}$. We proved in the last step that, for $p \geq 4$:
\begin{equation*}
    \left\{
        \begin{array}{ll}
        u_n &> 0 \qquad 0 \leq n < p_-,\\
        u_n &< 0 \qquad p_- < n \leq p.
        \end{array}
    \right.
\end{equation*}
Whether or not $p_-$ is an integer has no impact here. We deduce that at most one of the eigenvalues of the extracted matrix $T$ is non-positive by Lemma \ref{lemme_mineurs}. So, at most two eigenvalues of $T^{(j)}$ are non-positive, that is to say $\lambda_3 > 0$. By Proposition \ref{prop_eig1}, we obtain that $S^{(j)}$ is non-negative, which completes the proof of Lemma \ref{lemme_cas_impair}.
\end{proof}


We now turn to the case $j$ even.

\begin{lemme}\label{lemme_cas_pair}
	For all $j \geq 0$ even, we have $B^{(j)}_{| \mathcal{F}^{(j)}}$ non-negative.
\end{lemme}
\begin{proof}
	We fix an even integer $j$. We proceed as in the case $j$ odd above, but details are left to the reader.

{\underline    {Step 1 : Reduction to a symmetric matrix of lower order.}} Denote by $J= \text{antidiag}(1,\frac{j}{2})$. We follow again \cite{Cantoni-Butler} and Sect.~\ref{sec_centro}. The matrix $B^{(j)} \in \mathcal{M}_{j+1}(\R)$ can be written 
\begin{equation*}
    B^{(j)}  =   \begin{pmatrix}
    A & x & JCJ  \\[2pt]
    ^t x & q & ^t x J  \\[2pt]
    C &  Jx & JAJ                
\end{pmatrix}.
\end{equation*}
By Lemma \ref{lemme_CB_impair}, the eigenvalues of    $B^{(j)}_{| \mathcal{F}^{(j)}}$ are the eigenvalues of $S^{(j)}:= \begin{pmatrix} A+JC & \sqrt{2}x \\[2pt] \sqrt{2}\, ^tx & q \end{pmatrix}.$
A quick computation gives 
$$B^{(0)} = 0, \qquad S^{(0)} = 0,$$
\begin{equation*}
    B^{(2)}  = \frac14 \begin{pmatrix}
    -1 & 0 & 1\\[2pt]
    0 & 0& 0\\[2pt]    
    1 & 0&-1                
    \end{pmatrix}, \qquad S^{(2)}=0,
\end{equation*}
and
\begin{equation*}
    B^{(4)}  = \frac34 \begin{pmatrix}
    1 & -3 & 1& 1& 1\\[2pt]
    -3 & 1 & -1& 1& 1\\[2pt]
    1 & -1 & 1& -1& 1\\[2pt]  
    1 & 1 & -1& 1& -3\\[2pt]
    1 & 1 & 1& -3& 1
    \end{pmatrix},\qquad S^{(4)}=  \frac34 \begin{pmatrix}
    2 & -2 & \sqrt{2} \\[2pt]
    -2 & 2 & -\sqrt{2} \\[2pt]
    \sqrt{2} & -\sqrt{2} & 1
    \end{pmatrix} \geq 0.
\end{equation*}
Then, we can focus on the case $j \geq 6$ even.

For all $j\geq6$ (even), we can check that $0$ is an eigenvalue of $S^{(j)}$ of multiplicity at least $2$, and the eigenvectors associated are given by $v^{(j)}$ and $w^{(j)}$, where:
\begin{equation}\label{vecteurs_propres_B_pair1}
v_k^{(j)} = \frac{1}{(k-1)!(j-k+1)!}, \quad 1\leq k \leq q:=\frac{j}{2}, \qquad v_{q+1}^{(j)} = \frac{1}{\sqrt{2}q!^2};
\end{equation}
\begin{equation}\label{vecteurs_propres_B_pair2}
w_k^{(j)} = \frac{(k-1)(j-k+1)}{(k-1)!(j-k+1)!}, \quad 1\leq k \leq q, \qquad w_{q+1}^{(j)} = \frac{1}{\sqrt{2}(q-1)!^2}.
\end{equation}

We recall that, if we write $B^{(j)} = (B_{k,\ell}^{(j)})_{1\leq k,\ell\leq j+1}$, we have 
$$B_{k,k}^{(j)} = \frac{j!}{2^{j+1}} + \frac{1}{8}(j-4)(k-1)!(j-k+1)!, \qquad 1\leq k \leq j+1,$$
$$B_{k,k+1}^{(j)} = B_{k+1,k}^{(j)} = \frac{j!}{2^{j+1}} - \frac{1}{8}k!(j-k+1)!, \qquad 1\leq k \leq j,$$
and $B_{k,\ell}^{(j)} =\dfrac{j!}{2^{j+1}}$ otherwise.

We define $q \in \N$ by $j= 2q$, so that $S^{(j)} \in \mathcal{M}_{q+1}(\R).$ We are interested in the eigenvalues of $S^{(j)}$. The matrix $R^{(j)} = A+JC \in \mathcal{M}_q(\R)$ is the extraction of $S^{(j)}$ by erasing the last line and column, so the eigenvalues of $R^{(j)}$ intertwine those of $S^{(j)}$. We know that $0$ is an eigenvalue with multiplicity at least $2$ of $S^{(j)}$, so it suffices to show that $R^{(j)}$ is non-negative to obtain $S^{(j)}$ non-negative. Furthermore, we know that $0$ lies in the spectrum of $R^{(j)}$ since $0$ is an eigenvalue of $S^{(j)}$ with multiplicity at least $2$. We notice that $R^{(j)}$ can be written as the sum of a symmetric tridiagonal matrix and a matrix of rank unity. We write again $T^{(j)}$ this tridiagonal matrix, whose coefficients are given by :
\begin{align*}
    &T^{(j)}_{k,k} = \frac{1}{8}(2q-4)(k-1)!(2q+1-k)!  & 1 \leq k \leq q,\\
    &T^{(j)}_{k,k+1} = T_{k+1,k} = -\frac{1}{8}k!(2q+1-k)!& 1 \leq k \leq q-1.
\end{align*}


	{\underline    {Step 2 : Reduction to the study of a tridiagonal matrix.}}
The method in this step is identical as in the case $j$ odd. Since we only know that $0$ is an eigenvalue of $R^{(j)}$ with multiplicity (at least) $1$, the analogous of Proposition \ref{prop_eig1} is :

\begin{prop}\label{prop_eig2}
    Denote $\lambda_1 \leq \lambda_2 \leq \dots \leq \lambda_q$ the eigenvalues of $T^{(j)}.$ If $\lambda_2 > 0$, then $R^{(j)} \geq 0$.
\end{prop}

\medskip

	{\underline    {Step 3 : Study of a Sturm sequence.}} 
We study again the principal minors of the matrix $T^{(j)}-\lambda I_q$. This leads to the sequence
\begin{equation}\label{relat_vk}
	v_{k+2} = (2q-4)v_{k+1}-(k+1)(2q+1-(k+1))v_{k},\; k \geq 0, \qquad v_0=1, \;v_1=2q-4.
\end{equation}	

We then use Lemma~\ref{lemme_polynomial} with $a=2q-4$ and $b=2q+1$. We obtain directly the expression 
$$f(x) = \sum_{n=0}^{+\infty} \frac{v_{n}}{n!}x^{n} = (1+x)^{2q-2}(1-x)^{2}.$$
In particular, $v_n=0$ for all $n$ such that $n \geq 2q+1$. After computations, we get for every $n$ such that $2 \leq n \leq 2q-2$:
$$ v_n = 4\frac{(2q-2)!}{(2q-n)!}\pa{n-q_+}\pa{n-q_-},$$
where 
$$q_+ = \frac{1}{2}\pa{2q+\sqrt{2q}}, \qquad \text{and} \qquad q_- = \frac{1}{2}\pa{2q-\sqrt{2q}}.$$

	{\underline    {Step 4 : Conclusion.}}
We then have:
\begin{equation*}
    \left\{
        \begin{array}{ll}
        v_n &> 0 \qquad 0 \leq n < q_-,\\
        v_n &< 0 \qquad q_- < n \leq q.
        \end{array}
    \right.
\end{equation*}
Whether or not $q_-$ is an integer has no impact here. We deduce that at most one of the eigenvalues of the tridiagonal matrix $T^{(j)}$ is non-positive by Lemma~\ref{lemme_mineurs}, that is to say~$\lambda_2 >0$. By Proposition~\ref{prop_eig2}, we obtain that $R^{(j)}$ is non-negative, which completes the proof of Lemma~\ref{lemme_cas_pair}.
\end{proof}

\subsection{Conclusion}
In the last subsection, we proved that the functional~$B$ is non-negative (Proposition~\ref{Prop_B0}). Then, by Proposition~\ref{prop_Fmu}, Theorem~\ref{thm1}~$(ii)$ holds true.


\section{Proof of Theorem \ref{case_mu_12}}\label{section_cas_1_2}
In this section, we treat entirely the case $\mu = \frac12$. We give the complete description of global minimizers of $\mathcal{G}_\mu$. We start by giving some general result for any minimizer of $\mathcal{G}_\mu$.

\begin{prop}
    Assume that $u$ is a local minimizer of $\mathcal{G}_\mu$, for some $\mu > 0$. Then we have $Q(u) = 0$.
\end{prop}
\begin{proof}
    Let $u$ be a local minimizer for $\mathcal{G}_\mu = 8\pi \mathcal{H} + \mu P$ with $M(u)=1$. We consider $v_\alpha = R_\alpha u$, and compare $\mathcal{G}_\mu(u)$ and $\mathcal{G}_\mu(v_\alpha)$. We have $\mathcal{H}(v_\alpha) = \mathcal{H}(u).$ Furthermore, \eqref{action_trans_P} gives 
    $$ P(v_\alpha) = P(u) -2 \Re(\ov{\alpha} Q(u)) + |\alpha|^2 M(u).$$
    Since $u$ is a local minimizer, we have $P(v_\alpha) \geq P(u)$ for all $\alpha\in\C$ close enough to~$0$. Then, for the choice of $\dis \alpha = \frac{Q(u)}{M(u)}\eps,$ with $\eps >0$, we obtain 
    $$ P(v_\alpha) - P(u) = (\eps^2 - 2\eps) \frac{|Q(u)|^2}{M(u)}.$$
    For $\eps >0$ small enough, the quantity~$\eps^2-2\eps < 0,$ so that we obtain~$Q(u)=0$.
\end{proof}
\begin{remark}
    The magnetic momentum $Q(u)$ represents the distance of the function $u$ to the origin. Then, the latter proposition states that, in order to minimise $\mathcal{G}_\mu$, one should put the mass to be centred at the origin.
\end{remark}

We recall the expression 
$$ F_\mu(u) = 2E(u) + (\mu - \frac12)P(u)M(u),$$
and we recall that $F_\mu \geq 0$ means that $\phi_0$ is a global minimizer for~\eqref{min_pb}. We proved in Proposition~\ref{prop_BE}, Proposition~\ref{prop_Fmu}  and Proposition~\ref{Prop_B0} that $E \geq 0$, meaning $\phi_0$ is the unique global minimizer for any $\mu > \frac12$, and a global minimizer for $\mu = \frac12.$ Then, all other global minimizers for $\mu = \frac12$ are given by the equation $E(u) = 0$, of unknown $u \in L^{2,1}_\E$. In the remaining of the section, we prove:
\begin{theorem}\label{case_mu_12_bis}
	Suppose $\mu = \frac12$. Then, the global minimizers of \eqref{min_pb} are, up to phase and space rotations, $ \phi_0, \phi_1, \psi_b$.
\end{theorem}
\begin{proof}
    As said before, it suffices to study the equation $E(u) = 0$, of unknown $u\in L^{2,1}_\E$ with the constraint $M(u)=1$. For $u\in  L^{2,1}_\E$ such that $E(u)=0$, we have $B(u)=0$ (see Proposition~\ref{prop_mu_12} below). Let us begin by studying the equation $B(u)=0$:

\begin{lemme}\label{cas_egalite_B}
    Assume that $B(u)=0$, then there exists $a_0,a_1,c \in \C$ such that $$u(z)=(a_0 \phi_0(z)+a_1 \phi_1(z))e^{cz}.$$
\end{lemme}
\begin{proof}
    Let $u \in L^{2,1}_\E$ be a solution of $B(u)=0$. If $u$ has no zero, then we can write $u = C_1 \phi_0^c$, which gives $u(z)=a_0 \phi_0(z) e^{c z}$. Else, $u$ has at least one zero, say $z_0 \in \C$. We define $U(z) = R_{z_0}u(z) = u(z+z_0)e^{\frac12(\ov{z}z_0-z\ov{z_0})},$ which is $0$ for $z=0$. Then we can write $$U = \sum_{n=0}^{+\infty} a_n\phi_n,$$
    with $a_0=0$.
    Furthermore, note that we have $B(U) = B(u) = 0$ by Lemma~\ref{lemme_invariance_B}. As in Sect.~\ref{section_proof}, we set $A_{jk} = \frac{a_ka_{j-k}}{\sqrt{k!(j-k)!}}$ for $0 \leq k \leq j$. By Eq.~\eqref{exp_B}, and both Lemmas~\ref{lemme_cas_impair} and~\ref{lemme_cas_pair}, the assertion $B(U) = 0$ gives that, $\forall j \geq 0$, the vector $A_j = (A_{jk})_{0 \leq k \leq j}$ is an eigenvector of~$B^{(j)}$, associated to the eigenvalue~$0$. We recall that all other eigenvalues of~$B$ are positive. It gives (see Eqs.~\eqref{vecteurs_propres_B_impair}, \eqref{vecteurs_propres_B_pair1} and \eqref{vecteurs_propres_B_impair} above) that~$A_1$ proportional to~$v^{(1)}$, and for $j \geq 2$:
    $$ A_j = \alpha_j v^{(j)} + \beta_j w^{(j)}, \qquad \alpha_j,\beta_j \in \C,$$
    where $v^{(j)} = (v^{(j)}_k)_{0 \leq k \leq j}$, $w^{(j)} = (w^{(j)}_k)_{0 \leq k \leq j}$, and
    \begin{align*}
        v^{(j)}_k &= \frac{1}{k!(j-k)!}, \qquad 0 \leq k \leq j,\\
        w^{(j)}_k &= \frac{k(j-k)}{k!(j-k)!}, \qquad 0 \leq k \leq j.
    \end{align*}
    We set $b_k := \sqrt{k!}a_k, k\in\N$, to get the system of equations 
    \begin{equation}\label{systeme_bk}
        \left\{
            \begin{array}{ll}
            b_0b_1 &= \alpha_1, \\
            b_kb_{j-k} &= \alpha_j+k(j-k)\beta_j, \qquad  j \geq 2, \;\; 0 \leq k \leq j.
            \end{array}
        \right.
    \end{equation}
    From $a_0=0$, we have $b_0=0$, and $\alpha_1=0$. Then, for all $j \geq 2$, the choice of $k=0$ gives $b_jb_0 = 0 = \alpha_j.$ \newline
    \textbf{Case~$1$: $b_1 = 0$.} \\
    Suppose $b_1=0$. For $j \geq 2$, we use \eqref{systeme_bk} for $k=1$: $0 = b_{j-1}b_1 = (j-1)\beta_j,$ so $\beta_j=0$ for all $j\geq 2$, and from $b_j^2 = j^2\beta_{2j} = 0$ we obtain $b_j = 0$, that is to say $v=0$, so $u=0$. \newline 
    \textbf{Case~$2$: there exists $k_0\geq 2$, such that $b_{k_0} = 0$.} \\
    Suppose $b_{k_0}=0$ for some $k_0\geq 2$, and $b_\ell \neq 0$ if $ 1 \leq \ell \leq k_0-1.$ Then, for all $\ell \geq 1$, we have $0 = b_\ell b_{k_0} = k_0\ell\beta_{\ell+k_0}$, which means $\beta_j = 0$ for $j \geq k_0+1.$ If $\ell \geq k_0+1$, we obtain $b_\ell b_1 = \ell \beta_{\ell+1} = 0$, so that $b_\ell = 0$, since $b_1 \neq 0$. By hypothesis, $b_{k_0-1}^2 \neq 0$, and by \eqref{systeme_bk}, we have $b_{k_0-1}^2=(k_0-1)^2 \beta_{2(k_0-1)}$. But for $2(k_0-1) \geq k_0+1$, we have $\beta_{2(k_0-1)} = 0$, so that $k_0$ cannot satisfy $k_0 \geq 3$. Then $k_0 = 2$, which means that $v$ is proportional to $\phi_1$, and $u$ takes the form $u(z) = a_1\phi_1(z) e^{cz}$. \newline
    \textbf{Case~$3$: for all $j \geq 1, b_j \neq 0$.} \\
    Suppose that, for all $j \geq 1,$ we have $b_j \neq 0$. We then write, for all $j \geq 2,$ for all $k$ such that $ 1 \leq k \leq j-1$, 
    $$\beta_j = \frac{b_k}{k} \frac{b_{j-k}}{j-k},$$
    and $c_k := \frac{b_k}{k},$ to get
    \begin{equation}\label{eq_beta_j}
        \beta_j = c_kc_{j-k}, \qquad 1 \leq k \leq j-1.
    \end{equation}
    Since for all $j \geq 1, b_j \neq 0$, we have $\beta_j \neq 0,$ for $j \geq 1$. We use Eq.~\eqref{eq_beta_j} for both $k=1$ and $k=2$ to get 
    $$ \frac{c_{j-1}}{c_{j-2}} = \frac{c_2}{c_1}.$$ 
    Taking the product of this over $j$, we get 
    $$c_\ell = \lambda \alpha^\ell,$$
    for some $\alpha,\lambda \in \C.$ Then, 
    $$ a_\ell = \ell\frac{\lambda \alpha^\ell}{\sqrt{\ell!}},$$
    and we compute 
    \begin{equation*}
        U(z) = \sum_{k=1}^{+\infty} a_k \phi_k(z) = \lambda \sum_{k=1}^{+\infty} \frac{k \alpha^k}{\sqrt{k!}} \frac{z^k}{\sqrt{\pi k!}}e^{-\frac{|z|^2}{2}} = \frac{\lambda}{\sqrt{\pi}}\alpha z \pa{\sum_{k=1}^{+\infty} \frac{(\alpha z)^{k-1}}{(k-1)!} } e^{-\frac{|z|^2}{2}} = \lambda \alpha \phi_1(z) e^{\alpha z}.
    \end{equation*}
    Finally, by action of magnetic translation, $u$ takes the form $u(z) = (a_0\phi_0(z) + a_1\phi_1(z))e^{cz},$ and this completes the proof.
\end{proof}

\begin{prop}\label{prop_mu_12}
	Assume that $E(u)=0$ and $M(u) = 1$. Then, $B(u)=0$ and $Q(u) = 0$. It gives, modulo space and phase rotations, $u \in \{ \phi_0, \phi_1, \psi_b\}$.
\end{prop}
\begin{proof}
    By Proposition \ref{Prop_B0}, we have $$ 0 \leq B(u) + \frac{1}{4}|Q(u)|^2 = E(u) = 0,$$
    so that both $B(u)$ and $Q(u)$ vanish. By Proposition \ref{cas_egalite_B}, we can write $u(z)=(a_0 \phi_0(z)+a_1 \phi_1(z))e^{cz}$.\\
    If $a_1 = 0$, it reads $u= a_0 e^{\frac{|c|^2}{2}}\phi_0^c$, and $Q(u) = |a_0|^2e^{|c|^2} \ov{c}.$ Since $u \neq 0$ and $Q(u)=0$, we obtain $c= 0$, meaning that $u$ is proportional to $\phi_0$. \\
    Else, $a_1 \neq 0$, and by \cite[Proposition 6.2]{GGT}, there exist $(C,\phi,a,b) \in \C\backslash\{0\} \times \T \times \C \times \R$, such that 
    $$ u= CL_\phi R_a \psi_b.$$
    We compute, thanks to \eqref{action_trans_Q} and \eqref{action_rot_Q},
    $$ 0 = Q(u) = |C|^2Q(L_\phi R_a \psi_b) =  |C|^2e^{-i\phi}Q(R_a \psi_b) = |C|^2e^{-i\phi}\pa{Q(\psi_b) - aM(\psi_b)}.$$
    We recall that $Q(\psi_b) = 0$ and $M(\psi_b)=1$, giving $Q(u) = 0$ if and only if $a=0$. We also recall that for $b=0$, $\psi_b = \phi_1$. This completes the proof of Proposition~\ref{prop_mu_12}, and then the proof of Theorem~\ref{case_mu_12_bis}.
\end{proof}
\end{proof}


\section{Proof of Theorem \ref{min_phi1}}\label{section_cas_phi1}
In this section, we study the case $\mu \in [5/32,1/2)$. We prove the existence of some $\mu_0 \in [5/32,1/2)$ such that $\phi_1$ is a global minimizer of $\mathcal{G}_\mu$ for all $\mu \in (\mu_0,1/2)$. More precisely, we prove:
\begin{theorem}\label{min_phi1_bis}
	There exists $\mu_0 \in [5/32,1/2)$ such that, for all $\mu \in (\mu_0,1/2), \, \phi_1$ is the unique global minimizer, up to phase and space rotations, and for all $\mu < \mu_0$, any global minimizer of $\mathcal{G}_\mu$ has an infinite number of zeros.  
\end{theorem}

First, we need to prove:
\begin{lemme}\label{lemme_min_P}
    Consider $\mu_1 \in ]0,1/2[$ and $u\in \E$ such that 
    $$ \mathcal{G}_{\mu_1}(u) < \mathcal{G}_{\mu_1}(\phi_1).$$
    Then we have 
    $$ P(u) > P(\phi_1) = 1.$$
\end{lemme}
\begin{proof}
    We first note that, for a fixed function $v\in \E$, the application $\mu \mapsto \mathcal{G}_{\mu}(v)$ is a straight line, whose slope is $P(v)$ and with abscissa at the origin $8\pi \mathcal{H}(v).$ We also remark that
    \begin{equation}\label{eq:comparaison_Gmu}
        \mathcal{G}_{\frac12}(v) = \mathcal{G}_{\mu}(v) + \pa{\frac12 - \mu}P(v).
    \end{equation}
    Now, let us consider $\mu_1 \in ]0,1/2[$ and $u\in \E$ such that 
    $$ \mathcal{G}_{\mu_1}(u) < \mathcal{G}_{\mu_1}(\phi_1).$$
    By Theorem~\ref{case_mu_12_bis}, we know that $\phi_1$ is a global minimizer for $\mu=\frac12$, so that 
    $$ \mathcal{G}_{\frac12}(u) \geq \mathcal{G}_{\frac12}(\phi_1).$$
    By~\eqref{eq:comparaison_Gmu} applied for both $u$ and $\phi_1$ in $\mu=\mu_1$, we have
    $$ \mathcal{G}_{\mu_1}(u) + \pa{\frac12 - \mu_1}P(u) \geq \mathcal{G}_{\mu_1}(\phi_1) + \pa{\frac12 - \mu_1}P(\phi_1),$$
    which leads to 
    $$ \pa{\frac12 - \mu_1}\pa{P(\phi_1)-P(u)} \leq \mathcal{G}_{\mu_1}(u) - \mathcal{G}_{\mu_1}(\phi_1) < 0,$$
    so that $P(\phi_1)-P(u) < 0$ since $\mu_1 < \frac12$, and we conclude.
\end{proof}

    \begin{center}
    \begin{figure}[ht]
    \begin{tikzpicture}
    \begin{axis}[
    	  xlabel={$\mu$},
            ylabel={$\mathcal{G}_\mu$},
            axis on top=true,
            xmin=0,
            xmax=0.6,
            ymin=0.4,
            ymax=1.2,
            height=9.0cm,
            width=12.0cm,
            grid, 
    	 legend entries={},
    	 legend style={at={(0,1)},anchor=north west}]
        ]
            \draw [dashed, opacity=1] (axis cs:{0.501,0}) -- (axis cs:{0.501,1.2});    
            \draw [dashed, opacity=1] (axis cs:{0.15,0}) -- (axis cs:{0.15,1.2});
    	\addplot [domain=0:1, samples=80, mark=none, thick, green, name path=G0] {0.45+1.2*x};
    		\addlegendentry{$\mathcal{G}_\mu(u)$}
    	\addplot [domain=0:1, samples=80, mark=none, thick, red, name path=G1] {1/2+x};
    		\addlegendentry{$\mathcal{G}_\mu(\phi_1)$}
            \node [right, black] at (axis cs: 0.15,0.95) {$\mu = \mu_1$};
            \node [right, black] at (axis cs: 0.43,0.55) {$\mu = \frac12$};
    \end{axis}
    \end{tikzpicture}
    \caption{Plots of $\mathcal{G}_\mu(v)$ for $v=\phi_1$ and for a better candidate below $\mu=\frac12$.} \label{fig:graph2}
    \end{figure}
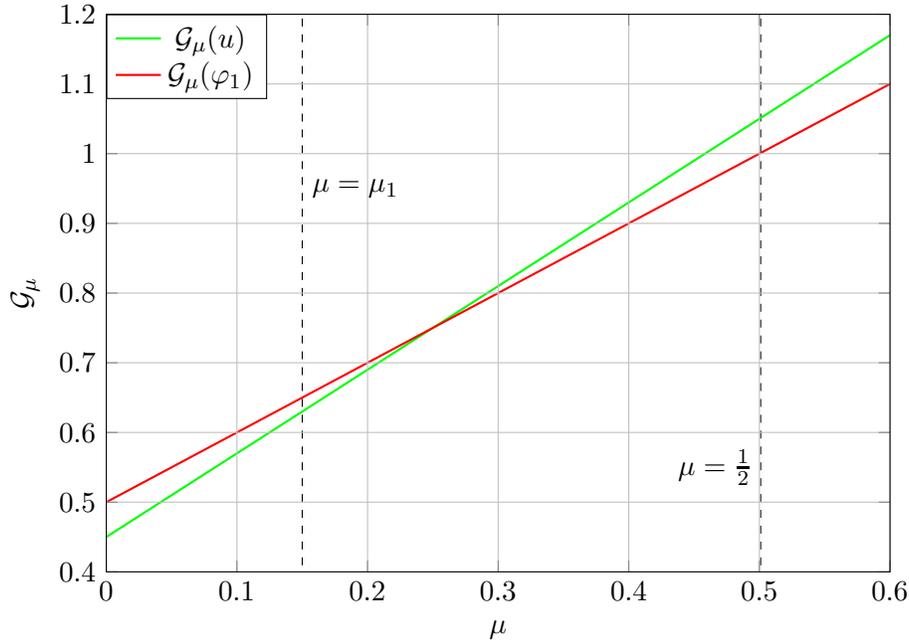
    \end{center}

\begin{remark}
    To illustrate Lemma~\ref{lemme_min_P} above, we plot on Fig.~\ref{fig:graph2}, $\mathcal{G}_\mu(\phi_1)$ and $\mathcal{G}_\mu(u)$ for a function~$u$ which is a better candidate than~$\phi_1$ at some $\mu = \mu_1 < \frac12$.
\end{remark}
\begin{remark}\label{rk_min_P}
    If we replace the hypothesis $\mathcal{G}_{\mu_1}(u) < \mathcal{G}_{\mu_1}(\phi_1)$ by $ \mathcal{G}_{\mu_1}(u) \leq \mathcal{G}_{\mu_1}(\phi_1),$ we obtain $P(u) \geq P(\phi_1)$. The proof is identical.
\end{remark}

We can now turn to the:
\begin{proof}[Proof of Theorem~\ref{min_phi1_bis}]
The proof is variational.
Denote by $(u_\mu)$ a sequence of global minimizers in $\{ u \in \mathcal{E}, M(u)=1 \}$ of $\mathcal{G}_\mu$ for any $\mu \in (5/32,1/2).$ Note that such a global minimizer for a fixed $\mu$ \emph{is not} unique, because of the symmetries of the equation. We have $M(u_\mu)=1$. From 
$$ 0 \leq \mathcal{G}_\mu(u_\mu) = 8\pi \mathcal{H}(u_\mu) + \mu P(u_\mu) \leq \mathcal{G}_\mu(\phi_0) = 1. $$
we get the existence of some absolute constant $C > 0$ such that $P(u_\mu) \leq C$ and $\mathcal{H}(u_\mu) \leq C$. By Carlen estimates~\eqref{eq_carlen}, $(u_\mu)$ is uniformly  bounded in~$\C$, and then also in $B(0,R)$ for all $R>0$. Writing $$u_\mu(z) = f_\mu(z) e^{-\frac{|z|^2}{2}}, \qquad f_\mu \, \text{holomorphic}, $$
implies that~$(f_\mu)$ is uniformly bounded in $B(0,R)$ for all $R>0$. By Montel's Theorem, we obtain a sequence~$(\mu_n)$ such that the sequence~$(f_{\mu_n})$ converges uniformly on all compacts of~$\C$ to a function~$f$ which is holomorphic on~$\C$, so that~$(u_{\mu_n})$ converges uniformly to $u := fe^{-\frac{|z|^2}{2}}.$ By Fatou's Lemma, we get $u \in \E.$ Since $M(u_{\mu_n}) = 1$, we can apply the concentration-compactness of~\cite{Lions1} to the sequence~$(|u_{\mu_n}|^2)$. We now show that we can only have compactness. First, the vanishing case is not possible. Indeed, we have for $R>1$, 
\begin{align*}
     P(u_\mu) &= \int_\C(|z|^2-1)|u_\mu(z)|^2 \dd L(z) \geq \int_{|z| \geq R}(R^2-1)|u_\mu(z)|^2\dd L(z) + \int_{|z| \leq 1}(|z|^2-1)|u_\mu(z)|^2\dd L(z) \\
     & \geq (R^2-1) \|u_\mu\|_{L^2(|z| \geq R)}^2 -1,
\end{align*}
which gives, using $P(u_\mu) \leq C$:
\begin{equation}\label{compactness_u_mu}
    \|u_\mu\|_{L^2(|z| \geq R)}^2 \lesssim \frac{1}{R^2}.
\end{equation}
We deduce that for any $0 < \eps < 1$, there exists $R_0 > 1$ satisfying 
$$ \sup_{y \in \C} \int_{B(y,R_0)} |u_{\mu_n}(z)|^2 \dd L(z) \geq \int_{B(0,R_0)} |u_{\mu_n}(z)|^2 \dd L(z) \geq 1-\eps,$$
and we cannot have vanishing. We can also show that the dichotomy case is not possible with similar arguments. Consequently, we are in the compactness case, and there exist $(z_k)$ and $(n_k)$ two sequences such that, for all $\eps >0$, there exists $R > 0$ satisfying
$$ \forall k \in \N, \quad \int_{B(z_k,R)}|u_{\mu_{n_k}}(z)|^2\dd L(z) \geq 1-\eps.$$
By Eq.~\eqref{compactness_u_mu}, we get that $(z_k)$ is bounded, so a sub-sequence converges to $\tilde{z} \in \C$. By extraction again, we have 
$$\int_{B(\tilde{z},R)}|u_{\mu_{m_k}}(z)|^2\dd L(z) \geq 1-\eps.$$
Using $u_{\mu_{m_k}} \to u \in \E$ uniformly on all compact sets of $\C$:
$$ 1 - \eps \leq \int_{B(\tilde{z},R)}|u_{\mu_{m_k}}(z)|^2\dd L(z) \to \| u \|^2_{L^2(B(\tilde{z},R))} \leq \| u \|^2_{L^2(\C)} = M(u),$$
and then $M(u) \geq 1$. By Fatou's Lemma, $M(u) \leq 1$, so that $M(u) = 1$. Consequently we can form a strictly increasing sequence $(\mu_n)$ of real numbers in $(5/32,1/2)$ such that $\mu_n \to 1/2$ and the sequence $(u_{\mu_n})$ converges in $L^2$ to $u\in \mathcal{E}$ such that $M(u)=1$. By Carlen estimates~\eqref{eq_carlen}, we have $\mathcal{H}(u_{\mu_n}) \to \mathcal{H}(u)$, and by lower semi-continuity of $P$, we obtain
\begin{align*}
    \mathcal{G}_{1/2}(u) &= 8\pi \mathcal{H}(u) + \frac12 P(u) = 8\pi \lim_{n \to +\infty}\mathcal{H}(u_{\mu_n}) + \pa{\liminf_{n \to +\infty} \mu_n} P(u) \\
    &\leq 8\pi \lim_{n \to +\infty}\mathcal{H}(u_{\mu_n}) + \pa{\liminf_{n \to +\infty} \mu_n} \pa{\liminf_{n \to +\infty}P(u_{\mu_n})} \leq \liminf_{n \to +\infty}\pa{\mathcal{G}_{\mu_n}(u_{\mu_n})} \\
    & \leq 1 = \mathcal{G}_{1/2}(\phi_0),
\end{align*}
where the last inequality comes from $\mathcal{G}_{\mu_n}(u_{\mu_n}) \leq \mathcal{G}_{\mu_n}(\phi_0) =1$ since $u_{\mu_n}$ is a global minimizer for $\mu = \mu_n$. Then $u$ is a global minimizer for $\mu=\frac12$, so we get $u \in \{ \phi_0, \phi_1, \psi_b\}$, modulo space and phase rotations, by Theorem~\ref{case_mu_12}. Now we aim at showing that $u$ is proportional to $\phi_1$. \\

$\bullet$ Let's consider $u_\mu \xrightarrow[\mu \to 1/2]{ } a_0\phi_0$ in $L^2(\C),$ for $\mu < 1/2$, and with mass $M(u_\mu)=1$. By action of phase rotation, we can assume that $a_0=1$. By Carlen estimates~\eqref{eq_carlen}, $u_\mu \xrightarrow[\mu \to 1/2]{ } \phi_0$ in $L^2(\C)$ implies $u_\mu \xrightarrow[\mu \to 1/2]{ } \phi_0$ in $L^4(\C),$ so that $\mathcal{H}(u_\mu) \to \mathcal{H}(\phi_0)$. From $\mathcal{G}_\mu(u_\mu) \to 1$, we deduce that 
\begin{equation*}
    P(u_\mu) \xrightarrow[\mu \to 1/2]{ } 0.
\end{equation*}
It means that $u_\mu \to \phi_0$ in $H^1$. There exists $\mu_1 \in (5/32,1/2)$ such that, for all $\mu \in (\mu_1,1/2),$ we have 
$$P(u_\mu) \leq \frac{1}{2} < 1 = P(\phi_1),$$
so $u_\mu$ cannot be global minimizer by Lemma~\ref{lemme_min_P} above. 

$\bullet$ Let's consider $u_\mu \xrightarrow[\mu \to 1/2]{ } C\psi_a$ in $L^2(\C),$ for $\mu < \frac{1}{2}, 0 \leq a \leq +\infty$, and with mass $M(u_\mu)=1$. By action of phase rotation, we can assume that $C=1$. By Carlen estimates~\eqref{eq_carlen}, $u_\mu \xrightarrow[\mu \to 1/2]{ } \psi_a$ in $L^2(\C)$ implies $u_\mu \xrightarrow[\mu \to 1/2]{ } \psi_a$ in $L^4(\C),$ so that $\mathcal{H}(u_\mu) \to \mathcal{H}(\psi_a)$. From $\mathcal{G}_\mu(u_\mu) \to 1$, we deduce that 
\begin{equation}\label{eq_lim_pa}
    P(u_\mu) \xrightarrow[\mu \to 1/2]{ } \frac{1}{1+a^2}.
\end{equation}
It means that $u_\mu \to \psi_a$ in $H^1$. If $a \neq 0$, then $\frac{1}{1+a^2} \in (0,1)$, and there is $\delta > 0$ such that $\frac{1}{1+a^2} + \delta < 1$. Due to the limit~\eqref{eq_lim_pa}, there exists $\mu_2 \in (5/32,1/2)$ such that, for all $\mu \in (\mu_2,1/2),$ we have 
$$P(u_\mu) \leq \frac{1}{1+a^2} + \delta < 1 = P(\phi_1),$$
so $u_\mu$ cannot be global minimizer by Lemma~\ref{lemme_min_P} above.

$\bullet$ We proved that the limit $u$ of $(u_n)$ must be proportional to $\phi_1$. We define $\mu_0$ as the infimum of $\mu$ such that $\phi_1$ is a global minimum. So far, $\mu_0$ could be equal to $\frac12$. Let us prove $\mu_0 < \frac12.$ Consider $u_\mu = \phi_1 + \chi_\mu \xrightarrow[\mu \to 1/2]{ } \phi_1$ in $L^2(\C)$, which corresponds to the case $u_\mu \to \psi_a$, with $a=0$. The function~$u_\mu$ is then a deformation of $\phi_1$, which is a \emph{strict} (up to the symmetries $L_\theta$ and $T_\gamma$) local minimizer for any $\mu \in (5/32,1/2)$, so that, for $\mu$ close enough to $1/2$, $u_\mu = e^{i\theta_\mu}\phi_1 = L_{\theta_\mu}\phi_1 = T_{\theta_\mu}\phi_1$, with $\theta_\mu \in \R.$ It means that $\mu_0 < \frac12$.\\
By Sect.~\ref{stats_waves}, for any $\mu < \mu_0$, no stationary waves with a finite number of zeros can be local minimizer besides $\phi_1$, so any global minimizer has an infinite number of zeros. \\
Let us now prove that $\phi_1$ is the \emph{unique} (up to symmetries) global minimizer if $\mu_0 < \mu < \frac12.$ Suppose there is $u\in \E$ with mass $M(u)=1$ such that $\mathcal{G}_{\mu_3}(u) = \mathcal{G}_{\mu_3}(\phi_1)$ for some $\mu_3 \in [\mu_0,\frac12)$. Then, by Remark~\ref{rk_min_P}, we have $P(u) \geq P(\phi_1).$ If $P(u) > P(\phi_1)$, we show $\mu_3 = \mu_0:$ for all $\eps > 0$, we have
$$ \mathcal{G}_{\mu_3-\eps}(u) = \mathcal{G}_{\mu_3}(u) + (\mu-\eps - \mu) P(u) = \mathcal{G}_{\mu_3}(\phi_1)- \eps P(u) < \mathcal{G}_{\mu_3}(\phi_1) - \eps P(\phi_1) = \mathcal{G}_{\mu_3-\eps}(\phi_1).$$ Then, 
$$ \forall \tilde{\mu} < \mu_3, \qquad \mathcal{G}_{\tilde{\mu}}(u) < \mathcal{G}_{\tilde{\mu}}(\phi_1),$$
which implies by definition of $\mu_0$ that $\mu_3 \leq \mu_0$. Else, suppose $P(u) = P(\phi_1)$, and $\mu_3 \in (\mu_0,\frac12).$ Then, the two straight lines $F_1 : \mu \mapsto \mathcal{G}_{\mu}(u)$ and $F_2 : \mu \mapsto \mathcal{G}_{\mu}(\phi_1)$ are parallel and take the same value in $\mu = \mu_3$, so they coincide on all points. In particular, $F_1(\frac{1}{2}) = F_2(\frac{1}{2})$, meaning that $u$ is also a global minimizer for $\mu=\frac12$. We conclude that $u = \phi_1$, up to phase and space rotations, with Theorem~\ref{case_mu_12_bis} and $P(u) = P(\phi_1)=1.$
\end{proof}


\section*{Statements and Declarations}
\subsection*{Data Availability}
Data sharing not applicable to this article as no datasets were generated or analysed during the current study.

\subsection*{Conflict of interest} 
The author certifies that he has no affiliations with or involvement in any organization or entity with any financial interest in the subject matter or materials discussed in this manuscript.


\end{document}